\documentclass[11pt]{amsart}
\usepackage{amsmath,amssymb,amsthm}
\usepackage{hyperref}
\usepackage{amsmath,amsfonts,amssymb,amsthm}
\usepackage{mathtools}
\usepackage{thmtools}

\newtheorem{lem}{Lemma}[section]
\newtheorem{teo}[lem]{Theorem}
 \newtheorem{Question}[lem]{Question}
 
\newtheorem{pro}[lem]{Proposition}
\newtheorem{cor}[lem]{Corollary}
\newtheorem{claim}[lem]{Claim}
\newtheorem{rem}[lem]{Remark}
\newtheorem{definition}[lem]{Definition}

\newtheorem*{teo*}{Theorem}

\newcounter{claimcounter}
\numberwithin{claimcounter}{lem}

\newcommand{\hookuparrow}{\mathrel{\rotatebox[origin=c]{90}{$\hookrightarrow$}}}

\newcommand{\myeq}[1]{\ensuremath{\stackrel{\text{#1}}{=}}}

\DeclareMathOperator{\D}{\mathcal D}

\DeclareMathOperator{\im}{Im}

\DeclareMathOperator{\Hom}{Hom}

\DeclareMathOperator{\rk}{rk}

\DeclareMathOperator{\Ann}{Ann}

\DeclareMathOperator{\Tor}{Tor}

\newcommand{\Z}{\mathbb{Z}}
\newcommand{\F}{\mathbb{F}}

\newcommand{\N}{\mathbb{N}}
\newcommand{\CC}{\mathbb{C}}
\newcommand{\Q}{\mathbb{Q}}

 
 
\title[Free  $\Q$-groups are residually torsion-free nilpotent]
{Free  $\Q$-groups are residually torsion-free nilpotent 
(Les $\Q$-groupes libres sont r\'esiduellement nilpotents sans torsion)}

\author{Andrei Jaikin-Zapirain}
 
\address{Departamento de Matem\'aticas, Universidad Aut\'onoma de Madrid \and  Instituto de Ciencias Matem\'aticas, CSIC-UAM-UC3M-UCM}

\email{andrei.jaikin@uam.es}

\begin{document}

\begin{abstract}
We develop a method to show that some  (abstract) groups can be  embedded into a free pro-$p$ group. In particular,  we show that  every finitely generated subgroup  of   a free $\Q$-group can be embedded into a  free pro-$p$ group for almost all primes $p$. This solves an  old problem raised by G. Baumslag: free $\Q$-groups are residually torsion-free nilpotent.

Nous d\'eveloppons une m\'ethode pour montrer que certains groupes (discrets) peuvent \^etre plong\'es dans un pro-$p$-groupe libre. Nous montrons en particulier  que tout sous-groupe de type fini d'un $\Q$-groupe libre peut \^etre plong\'e dans un pro-$p$-groupe libre pour presque tous les premiers $p$. Cela r\'esout un ancien probl\`eme soulev\'e par G. Baumslag: les $\Q$-groupes libres sont  r\'e'siduellement nilpotents sans torsion.
\end{abstract}

  \maketitle

 \section{Introduction} 
 A  group $G$ is called a {\bf $\Q$-group } if for any $n\in \N$ and $ g\in G$  there exists exactly one $h\in G$ satisfying $h^n=g$. These groups were introduced by G. Baumslag in \cite{Ba60} under the name of $\D$-groups. He observed that $\Q$-groups may be viewed as universal algebras, and as such they constitute
a variety. Every variety of algebras contains free algebras (in that variety).
In the variety of  $\Q$-groups we  call  such free algebras {\bf free $\Q$-groups}.  G. Baumslag dedicated several papers to the study of residual properties of free $\Q$-groups \cite{Ba65, Ba68, Ba10}. For example, in \cite{Ba65} he showed that a free $\Q$-group is residually periodic-by-soluble and locally residually finite-by-soluble. He wrote in \cite{Ba65} ``It is, of course, still possible that, locally, free $\D$-groups are, say, residually
finite $p$-groups'' or in \cite{Ba68} ``In particular it seems likely that free $\D$-groups are residually torsion-free nilpotent. However the complicated nature of free $\D$-groups makes it difficult to
substantiate such a remark.'' This conjecture   is part of two main collections of problems in group theory (\cite[Problem F12]{BMS} and \cite[Problem 13.39 (a),(c)] {Kou}), and in addition to mentioned works of Baumslag, it was  also studied in  \cite{Ch68, GMRS97}.
 In this paper we solve Baumslag's conjecture.
 
 \begin{teo} \label{teoBau} A free $\Q$-group is  residually torsion-free nilpotent.
 \end{teo}
The structure of a finitely generated subgroup of a free $\Q$-group was studied already in \cite{Ba60} (see also \cite[Section 8]{MR96} and Proposition \ref{MR}). It was shown that it is  the end result of repeatedly freely adjoining $n$th roots to a finitely generated  free group. 
The key point of our   proof of Theorem \ref{teoBau} is to show  that any finitely generated subgroup  of a free $\Q$-group can be embedded into a finitely generated free pro-$p$ group for some prime $p$. 
We actually  prove the following more precise result.
\begin{teo}\label{iteratedextension} Let $p$ be a prime.
Let $H_0$ be a finitely generated free group  and let $H_0\hookrightarrow \mathbf F$ be the canonical embedding of $H_0$ into its pro-$p$ completion $ \mathbf F$. Let $(H_{i})_{i\ge 0}$ be a sequence of subgroups of  $\mathbf F$ such that for $i\ge 0$,
\begin{enumerate}
\item $H_{i+1}=\langle H_i, B_i\rangle $, where   $B_{i}$ is a finitely generated  abelian subgroup of $\mathbf F$ and 
\item $A_i=H_i\cap B_i$ is  a maximal abelian subgroup of $H_{i}$.
\end{enumerate}
Then  for every $i\ge 0$, the canonical map $$H_{i}*_{A_{i}}B_{i}\to H_{i+1}$$ is an isomorphism,
\end{teo}
Theorem \ref{iteratedextension} is actually an application of the slightly more technical Theorem \ref{main}.

Let us make a few remarks about the groups $A_i$ and $B_i$.
It is relatively easy to describe  abelian subgroups of amalgamated products. In particular,  the conclusion of  the theorem  implies  that all  abelian subgroups of $H_i$ are finitely generated. Thus, an implicit hypothesis, which appears in the theorem,  that maximal abelian subgroups $A_i$ of $H_i$ are finitely generated, is automatically fulfilled.

A maximal abelian subgroup of $\mathbf F$  is isomorphic to the additive group of the ring of $p$-adic numbers $(\Z_p,+)$. Therefore, for any finitely generated (abstract) abelian subgroup $A$ of $\mathbf F$  and any finitely generated torsion-free abelian group $B$ which contains $A$ and such that $B/A$ has no $p$-torsion, it is posible to extend the embedding  $A\hookrightarrow \mathbf F$ to an embedding $B\hookrightarrow \mathbf F$. This extension is   unique if and only if   $B/A$ is finite.

Given a commutative ring $A$, we will introduce in Section \ref{proofs}  the notions of $A$-group and free $A$-group $F^A(X)$. For example, a free pro-$p$ group is an example of a $\Z_p$-group. 
We have the following  consequence of Theorem \ref{iteratedextension}.
\begin{cor}\label{Zpgroup}
 Let $F(X)$ be the free group on a finite free generating set $X$, let $\mathbf F$ be its pro-$p$ completion. Then the canonical homomorphism
$$\phi:F^{\Z_p}(X)\to \mathbf F$$ is injective.
\end{cor}

Let $H$ be a group and
$A$ the centralizer of a non-trivial  element. Then the group $G = H*_A(A\times \Z^k)$ is said to be obtained from $H$  by {\bf extension of a centralizer}. 
A group is called an {\bf ICE}  group if it can be obtained from a free group using   iterated centralizer extensions. A group $G$ is a {\bf limit group} if and only if it is a finitely generated subgroup of an ICE group  (see \cite{KM06, CG05}). All centralizers of non-trivial elements of an ICE group are abelian.
Thus, Theorem \ref{teoBau} provides explicit realizations   of ICE groups (and so limit groups) as subgroups of a non-abelian free pro-$p$-groups (for this application we only need the case where all $B_i/A_i$ are torsion-free). Non-explicit realizations of limit groups as subgroups of a non-abelian free pro-$p$ group (in fact, as subgroups of every compact group containing a non-abelian free group) was obtained in \cite{BG10} (see also \cite{BGSS06}).

In Section \ref{proofs} we recall the definition of the $\Q$-completion of a group $G$. For example, a free $\Q$-group is the $\Q$-completion of a free group. Theorem \ref{iteratedextension} allows also to show that the $\Q$-completion of a limit group is residually torsion-free nilpotent.
\begin{teo}\label{Qlimit}
The $\Q$-completion of a limit group is residually torsion-free nilpotent.
\end{teo}
A group $G$ is called {\bf parafree}  if   it is residually nilpotent and for some free group $F$, we have that for all $i$, $G/\gamma_i(G)\cong F/\gamma_i(F)$ where $\gamma_i(G)$ denotes the terms of the lower central series of $G$.  Baumslag introduced this family of groups and  produced many examples of them \cite{Ba67}.
In \cite{JM21} we apply the method of the proof of  Theorem \ref{iteratedextension} in order to construct new examples of finitely generated parafree groups.

Our proof of Theorem 1.2 is by induction on $i$. In the inductive step argument we start with the following situation. We have  a finitely generated subgroup $H$ of $\mathbf F$, a maximal abelian subgroup $A$ of $H$ and an abelian subgroup $B$ of $\mathbf F$ containing $A$. We want to show that the canonical homomorphism $H*_AB\to\langle H,B\rangle$ is an isomorphism. Unfortunately, we do not know how to show this statement  in such a generality, but we prove it in Theorem \ref{main} under an  additional assumption  that the embedding $H\hookrightarrow \mathbf F$ is strong (see Definition \ref{defsep}). Theorem \ref{main} is the main result of the paper. Its proof uses in an essential way the results of \cite{Jataylor19}, where we proved a particular case of the L\"uck approximation in positive characteristic.

 The paper is organized as follows. In Section \ref{prelim} we give basic preliminaries.
 The proof of  Theorem \ref {main} uses the theory of   mod-$p$ $L^2$-Betti numbers. In Section \ref{betti} we explain how to define them for subgroups  $G$ of a free pro-$p$ group. In Section \ref{tech} we introduce a technical notion of $\mathcal D$-torsion-free modules and show that some relevant $\F_p[G]$-modules are $\D_{\F_p[G]}$-torsion-free (see Proposition \ref{dfpfree}). In Section \ref{proofs} we prove Theorem \ref{main} and obtain  all the results mentioned in the introduction. In Section \ref{linearity}
 we discuss the following  two well-known problems  concerning linearity of free pro-$p$ groups and free $\Q$-groups:
 
 \begin{Question} \begin{enumerate}
 \
  \item (I. Kapovich) Is a free $\Q$-group linear?
 \item (A. Lubotzky)  Is a free pro-$p$ group linear?

 \end{enumerate}
\end{Question}

\section*{Acknowledgments}

I am very grateful to Pavel Zalesski who explained to me how   a surface group can be embedded  into a free pro-$p$ group. This was a starting point of Theorem \ref{iteratedextension}.
I would  like to thank Warren Dicks for providing the references needed in the proof of Proposition \ref{critamal}, Yago Antolin, Ashot Minasyan and a referee for explaining to me the proof of Theorem \ref{linear}, Emmanuel Breulliard for pointing out  the paper \cite{BG10} and Ismael Morales and anonymous  referees for  useful  suggestions and comments.  I am grateful to Aleksandr Krasilnikov for spotting a mistake in the previous version of the proof of Theorem \ref{MagnusA}.

This paper is partially supported by the grants   MTM2017-82690-P and PID2020-114032GB-I00 of the Ministry of Science and Innovation of Spain  and by the ICMAT Severo
Ochoa project  CEX2019-000904-S4.

 \section{Preliminaries}\label{prelim}  
\subsection{$R$-rings}
All rings in this paper are associative   and  have the identity element.  All   ring homomorphisms send  the identity   to the identity. We denote the invertible elements of  a ring $R$ by $R^*$. An $R$-module   means a left $R$-module.
By an {\bf $R$-ring} we understand a  ring homomorphism $\varphi: R\to S$. We will often refer to $S$ as $R$-ring and omit the homomorphism $\varphi$ if $\varphi$ is clear from the context.
Two $R$-rings $\varphi_1:R\to S_1$ and $\varphi_2:R\to S_2$ are said to be {\bf isomorphic} if there exists a ring isomorphism $\alpha: S_1\to S_2$ such that $\alpha\circ \varphi_1=\varphi_2$.

If $Y=\{y_i\colon i\in I\}$, we denote by   $A\langle \! \langle  Y\rangle \! \rangle $ the ring of   of formal
power series with coefficients in $A$ on the non-commuting  indeterminates $Y$. 
\subsection{Left ideals in  group algebras}

Let $G$ be a group and $k$ a commutative ring. We denote by $I_G$ the augmentation ideal of $k[G]$. If $H$ is a subgroup of $G$ we denote by $I_H^G$ the left ideal of $k[G]$ generated by $I_H$. The following lemma gives an alternative description of the $k[G]$-module $I_H^G$.

\begin{lem} \label{isomdiscr} Let $H\le T$ be      subgroups of $G$. Then the following holds.
\begin{enumerate}
\item[(a)] The canonical map $$k[G]\otimes_{k[H]} I_H\to I_H^G$$ sending  $a\otimes b$ to $ab$, is an isomorphism of $k[G]$-modules.

\item [(b)] The canonical map  $k[G]\otimes_{k[T]} (I_T/I_H^T) \to I_T^G/I_H^G $, sending $a\otimes (b+I_H^T)$ to $ab+I_H^G$, is an isomorphism of $k[G]$-modules.

\end{enumerate}

\end{lem}
\begin{proof}
(a) Consider an exact sequence
$$0\to I_H\to k[H]\to k \to 0.$$
The freeness of $k[G]$ as  $k[H]$-module implies that the sequence 
$$0\to k[G] \otimes_{k[H]} I_H\xrightarrow{\alpha}k[G]\xrightarrow{\beta} k[G] \otimes_{k[H]} k\to 0$$
 is also exact.  Here $\alpha$  sends $a\otimes b$ to $ab$ and $\beta$ sends $a$ to $a\otimes 1$. Thus, $\alpha$ establishes an isomorphisms of $k[G]$-modules
between $k[G]\otimes_{k[H]} I_H$ and $\ker \beta= I_H^G$.     This proves the first claim of the lemma.
 
 (b) Consider now the exact sequence
$$0\to I_H^T\to I_T\to I_T/I_H^T\to 0.$$ Applying $k[G] \otimes_{k[T]}$ and taking   again into account that $ k[G]$ is a free   $k[T]$-module, we obtain the exact sequence

$$0\to I_H^G\to I_T^G\to k[G] \otimes_{k[T]}( I_T/I_H^T)\to 0.$$ 
 This proves the second claim.
 \end{proof}
 
 \subsection{Profinite modules over pro-$p$ groups}

In this paper the letters $\mathbf F$, $\mathbf G$, $\mathbf H$, etc. will denote    pro-$p$ groups. When we speak about pro-$p$ groups, the finite generation,  the finite presentation, the freeness etc.  will always be considered in the category of pro-$p$ groups. For example, $d(\mathbf G)$ denotes the minimal number of topological generators of $\mathbf G$.

Almost all pro-$p$ groups that we consider are free pro-$p$ groups. Recall that a closed subgroup of a free pro-$p$ group is also free pro-$p$ (\cite[Corollary 7.7.5]{RZ10}). As a consequence we obtain the following result which we will  use often in this paper.

\begin{lem} \label{naxab}
Every maximal abelian subgroup of a non-trivial free pro-$p$ group is isomorphic to $(\Z_p,+)$.
\end{lem}

Let $\mathbf G$ be a  pro-$p$ group. We denote by $\F_p[[\mathbf G]]$ the inverse limit of $\F_p[\mathbf G/\mathbf U]$, where the limit is
   taken over     all open normal subgroups $\mathbf U$ of $\mathbf G$.  $\F_p[[\mathbf G]]$ is called the {\bf   completed group algebra of $\mathbf G$ over $\F_p$}. In the case where $\mathbf G$ is a free pro-$p$ group we have the following useful description of  $\F_p[[\mathbf G]]$.
   \begin{pro}\cite[Section II, Proposition 3.1.4]{La65} \label{Lazard}
   Let $\mathbf F$ be a finitely generated free pro-$p$ group freely generated by $x_1,\ldots, x_d$. Then the continuous $\F_p$-algebra homomorphism
   $$\F_p\langle \!\langle y_1,\ldots, y_d\rangle \!\rangle \to \F_p[[\mathbf F]]$$ that sends $y_i$ to $ x_1-1$  (for $1\le i\le d$) is an isomorphism.
   \end{pro}
   
   A {\bf discrete $\F_p[[\mathbf G]]$-module} is an $\F_p[[\mathbf G]]$-module $M$ such that for any $m\in M$, $$\Ann_{\F_p[[\mathbf G]]}(m):=\{a\in \F_p[[\mathbf G]]\colon am=0\}$$ is open in $\F_p[[\mathbf G]]$. 
   
     A {\bf profinite $\F_p[[\mathbf G]]$-module} is an inverse limit of finite discrete $\F_p[[\mathbf G]]$-modules.
    \begin{lem}\label{continuous} 
    Let $\mathbf G$ be a  pro-$p$ group.
   Let $\alpha: M\to N$ be a homomorphism of profinite $\F_p[[\mathbf G]]$-modules. If $M$ is finitely generated as an $\F_p[[\mathbf G]]$-module, then $\alpha$ is continuous.
   \end{lem}
   \begin{rem}
   This lemma resembles  a well-known result of Nikolov and Segal \cite{NS07} that says that every  homomorphism from a finitely generated profinite group to a profinite group is always continuous. This is   equivalent to  that every subgroup of finite index in a finitely generated profinite group is open. In the case of $\F_p[[\mathbf G]]$-modules, the situation is different:  it is not always true that every left ideal of $\F_p[[\mathbf G]]$ of finite index is open.
   \end{rem}
   \begin{proof} Without loss of generality we may  assume that $N$ is a finite discrete $\F_p[[\mathbf G]]$-module and we have to show that $\ker \alpha $ is open. 
   
   We put  $I=\Ann_{\F_p[[G]]}=\{a\in \F_p[[\mathbf G]]\colon aN=\{0\}\}$. Then $I$ is an open, and so also closed,  ideal of $\F_p[[\mathbf G]]$. Let $IM$ be the submodule of $M$ generated by $\{a\cdot m\colon a\in I, m\in M\}$. Since $M$ is finitely generated we can write $$M=\sum_{i=1}^s \F_p[[\mathbf G]]m_i.$$ Therefore,  $IM=\sum_{i=1}^s Im_i$.
   
   From the definition of a profinite $\F_p[[\mathbf G]]$-module it follows that the multiplicative map $\F_p[[\bf G]]\times M\to M$ is continuous. Hence, since $I$ is closed in $\F_p[[\mathbf G]]$,
$IM=\sum_{i=1}^s Im_i$ is a closed submodule of $M$. But $IM$ is also of finite index. Hence, it  is open in $M$. Since $IM\le \ker \alpha$, $\ker \alpha$ is also open.   
   \end{proof}

   If $M=\varprojlim_{i\in I} M_i$ and $N=\varprojlim_{j\in J} N_i$ are right and left, respectively,  profinite $\F_p[[\mathbf G]]$-modules ($M_i$ and  $N_j$ are finite discrete $\F_p[[\mathbf G]]$-modules), then the profinite tensor product is denoted by $\widehat \otimes$ and it is defined as the inverse limit of $M_i\otimes_{\F_p[[\mathbf G]]}N_j$.
\begin{lem}\label{profinitetensor}
Let 
 $\mathbf H$ be  a closed subgroup of $\mathbf G$ and let $M$ be a finitely presented $\F_p[[\mathbf H]]$-module. Then $M$ is a profinite module and  the canonical map $$\gamma:\F_p[[\mathbf G]]\otimes_{\F_p[[\mathbf H]]}M\to \F_p[[\mathbf G]]\widehat \otimes_{\F_p[[\mathbf H]]} M$$ is an isomorphism of $\F_p[[\mathbf G]]$-modules.
 \end{lem}
 \begin{rem}
 It is claimed in \cite[Proposition 5.5.3 (d)]{RZ10}   that it is enough to assume that $M$ is a  finitely generated profinite $\F_p[[\mathbf H]]$-module. However, we want to warn  the reader that the proof  is incorrect.
 \end{rem}
 \begin{proof} Since $M$ is finitely presented, there exists an exact sequence of $\F_p[[\mathbf H]]$-modules
\begin{equation}\label{sequence}
 \F_p[[\mathbf H]]^r\xrightarrow{\alpha} \F_p[[\mathbf H]]^d\to M\to 0.
 \end{equation}
  Thus, we can write  $M\cong \F_p[[\mathbf H]]^d/I$, where $I=\im \alpha$. By Lemma \ref{continuous},  $\alpha$ is continuous.   Since  $\F_p[[\mathbf H]]^r$ is compact and Hausdorff, $I$ is closed in  $\F_p[[\mathbf H]]^d$. Hence  $\F_p[[\mathbf H]]^d/I$, and so $M$, are profinite.

 It is clear that 
$$\F_p[[\mathbf G]]\widehat \otimes_{\F_p[[\mathbf H]] } \F_p[[\mathbf H]]\cong \F_p[[\mathbf G]] \otimes_{\F_p[[\mathbf H]] }\F_p[[\mathbf H]]\cong \F_p[[\mathbf G]] .$$ Thus,
after applying $\F_p[[\mathbf G]]\otimes _{\F_p[[\mathbf H]] }$ and  $\F_p[[\mathbf G]]{\widehat \otimes} _{\F_p[[\mathbf H]] }$ to the sequence (\ref{sequence}) we obtain a commutative diagram of $\F_p[[\mathbf G]]$-modules:
$$\begin{array}{cccccc}
 \F_p[[\mathbf G]]^r & \to &  \F_p[[\mathbf G]]^d & \to&  \F_p[[\mathbf G]]\otimes _{\F_p[[\mathbf H]]} M &\to 0\\
 ||&&||&&\downarrow^\gamma&\\
  \F_p[[\mathbf G]]^r & \to &  \F_p[[\mathbf G]]^d & \to&  \F_p[[\mathbf G]]\widehat \otimes _{\F_p[[\mathbf H]]} M &\to 0.
\end{array}.
$$
Since $\otimes$ and $\widehat \otimes $ are right  exact (see \cite[Theorem 2.6.3]{Ro09} and \cite[Proposition 5.5.3(a)]{RZ10}) the horizontal sequences are exact.  This clearly implies that $\gamma$ is an isomorphism. \end{proof}
 
 Recall that if a pro-$p$ group $\mathbf G$  is finitely generated, then the augmentation ideal $I_{\mathbf G}$ is finitely generated as an $\F_p[[\mathbf G]]$-module. In particular, the trivial $\F_p[[\mathbf G]]$-module $\F_p$ is finitely presented. We can extend this to all  finite discrete $\F_p[[\mathbf G]]$-modules. 
   \begin{lem} \label{finitediscrete}
   Let $\mathbf G$ be a finitely generated pro-$p$ group and let $M$ be a  finite discrete $\F_p[[\mathbf G]]$-module. Then $M$ is   finitely presented as an $\F_p[[G]]$-module.
   \end{lem}
   \begin{proof} We have to show that every open submodule  of a finitely generated profinite $\F_p[[\mathbf G]]$-module $M$ is finitely generated. Arguing by induction on the index of the open submodule, we easily obtain this statement from the fact that $I_{\mathbf G}$ is finitely generated.
   \end{proof}

  \subsection{Left ideals in    completed group algebras}
  Now we apply the results of the previous subsection in a particular situation that interests us.
 
Let $\mathbf G$ be a  pro-$p$ group.     
 The definition of free profinite $\F_p[[\mathbf G]]$-modules can be found in \cite[Section 5.2]{RZ10}. We will  need the the following fact.
 
 \begin{lem}\label{freemodule}
 Let $\mathbf G$ be a  pro-$p$ group and $\mathbf H$ a closed subgroup of $\mathbf G$. Then $\F_p[[\mathbf G]]$ is a free profinite $\F_p[[\mathbf H]]$-module. In particular the functor $\F_p[[\mathbf G]]\widehat \otimes_{\F_p[[\mathbf H]]}-$ is exact.
 \end{lem}
 \begin{proof}
 The first part follows from \cite[Corollary 5.7.2]{RZ10} and the second  one from \cite[Proposition 5.5.3 (e)]{RZ10}. 
 \end{proof}
We denote by $I_\mathbf G$ the augmentation ideal of $\F_p[[\mathbf G]]$.  The following lemma is well-known. We provide a proof for the convenience of the reader.
\begin{lem}\label{finitepresent}
Let $\mathbf H$ be a finitely  presented   pro-$p$ group. Then $I_\mathbf H$ is finitely presented as a $\F_p[[\mathbf H]]$-module.
\end{lem}
\begin{proof} Since $\F_p[[\mathbf H]]$ is a local ring, a profinite $\F_p[[\mathbf H]]$-module $M$ is finitely presented  if and only if $H_0(\mathbf H; M)$ and $H_1(\mathbf H, M)$ are finite. Since $\mathbf H$ is finitely presented as a pro-$p$ group, \cite[Theorem 7.8.1 and Theorem 7.8.3]{RZ10} imply that $H_1(\mathbf H; \F_p)$ and $H_2(\mathbf H; \F_p)$ are finite. However, we have that for $i\ge 0$, $H_i(\mathbf H, I_{\mathbf H})=H_{i+1}(\mathbf H;\F_p)$. Thus, $I_\mathbf H$ is finitely presented as a $\F_p[[\mathbf H]]$-module.
\end{proof}

If $\mathbf H$ is a closed subgroup of $\mathbf G$, then $I_\mathbf H^\mathbf G$ denotes the closed left ideal of $\F_p[[\mathbf G]]$ generated by $I_\mathbf H$. 

\begin{lem} \label{usisom} Let $\mathbf G$ be a pro-$p$ group and let $\mathbf H\le \mathbf T$ be   closed subgroups of $\mathbf G$. Then the following holds.
\begin{enumerate}
\item [(a)] The   continuous map $$\displaystyle \F_p[[\mathbf G]]\widehat \otimes_{\F_p[[\mathbf H]]} I_\mathbf H\to I_\mathbf H^\mathbf G$$ that sends $a\otimes b$ to $ab$, is an isomorphism of $\F_p[[\mathbf G]]$-modules.

\item [(b)] If $\mathbf H$ is finitely presented, the  map $$\displaystyle \F_p[[\mathbf G]]\otimes_{\F_p[[\mathbf H]]} I_\mathbf H\to I_\mathbf H^\mathbf G$$ that sends $a\otimes b$ to $ab$, is an isomorphism of $\F_p[[\mathbf G]]$-modules.
\item [(c)] If $\mathbf T$ is finitely presented and $\mathbf H$ is finitely generated, the   map $$\displaystyle \F_p[[\mathbf G]]\otimes_{\F_p[[\mathbf T]]} (I_\mathbf T/I_\mathbf H^\mathbf T) \to  I_\mathbf T^\mathbf G/I_\mathbf H^\mathbf G$$
  sending $a\otimes (b+I_\mathbf H^\mathbf T)$ to $ab+I_\mathbf H^\mathbf G$,  is an isomorphism of $\F_p[[\mathbf G]]$-modules.

\end{enumerate}

\end{lem}
\begin{proof}
(a) Consider an exact sequence
$$0\to I_\mathbf H\to\F_p[[\mathbf H]]\to\F_p \to 0.$$
By Lemma  \ref{freemodule}, the sequence 
$$0\to\F_p[[\mathbf G]]\widehat \otimes_{\F_p[[\mathbf H]]} I_\mathbf H\xrightarrow{\alpha}\F_p[[\mathbf G]]\xrightarrow{\beta}\F_p[[\mathbf G]]\widehat \otimes_{\F_p[[\mathbf H]]}\F_p\to 0$$
 is also exact.  
Thus,  $\F_p[[\mathbf G]]\widehat \otimes_{\F_p[[\mathbf H]]} I_\mathbf H$ is isomorphic to $\ker \beta=I_\mathbf H^\mathbf G$. 

(b) Since $\mathbf H$ is finitely presented, by Lemma \ref{finitepresent}, $I_\mathbf H$ is finitely presented as $\F_p[[\mathbf H]]$-module, and so, by Lemma \ref{profinitetensor},  
$$\F_p[[\mathbf G]]\widehat \otimes_{\F_p[[\mathbf H]]} I_\mathbf H\cong \F_p[[\mathbf G]] \otimes_{\F_p[[\mathbf H]]} I_\mathbf H.$$ 
Thus we conclude that the natural map from  $\F_p[[\mathbf G]]\otimes_{\F_p[[\mathbf H]]} I_\mathbf H$ to $I_\mathbf H^\mathbf G$ is also an isomorphism by (a).

(c) Consider now the exact sequence
$$0\to I_\mathbf H^\mathbf T\to I_\mathbf T\to I_\mathbf T/I_\mathbf H^\mathbf T\to 0.$$ Applying $\F_p[[\mathbf G]]\widehat \otimes_{\F_p[[\mathbf T]]}$ and taking   again into account that $\F_p[[\mathbf G]]$ is a free profinite $\F_p[[\mathbf T]]$-module (Lemma \ref{finitepresent}), we obtain the exact sequence

$$0\to I_\mathbf H^\mathbf G\to I_\mathbf T^\mathbf G\to\F_p[[\mathbf G]]\widehat \otimes_{\F_p[[\mathbf T]]}( I_\mathbf T/I_\mathbf H^\mathbf T)\to 0.$$ 
Since $\mathbf T$ is finitely presented and $\mathbf H$ is finitely generated,  by Lemma \ref{finitepresent},  $I_\mathbf T/I_\mathbf H^\mathbf T$ is finitely presented  as $\F_p[[\mathbf T]]$-module. 
Thus,  by Lemma \ref{profinitetensor}, the canonical map $$ \F_p[[\mathbf G]]\otimes_{\F_p[[\mathbf T]]}( I_\mathbf T/ I_\mathbf H^\mathbf T)\to \F_p[[\mathbf G]]\widehat \otimes_{\F_p[[\mathbf T]]}( I_\mathbf T/I_\mathbf H^\mathbf T) $$ is an isomorphism.  This proves the last claim.
 \end{proof}
\begin{cor}\label{freemod} Let $\mathbf H$ be a finitely generated free subgroup of a pro-$p$ group $\mathbf G$. Then $I_{\mathbf H}^{\mathbf G}$ is a free $\F_p[[\mathbf G]]$-module  of rank $d(\mathbf H)$.
\end{cor}
\begin{proof}
By Proposition \ref{Lazard}, $I_{\mathbf H}$ is a free $\F_p[[\mathbf H]]$-module of rank $d(\mathbf H)$. Therefore, by Lemma \ref{usisom}(b),  $I_{\mathbf H}^{\mathbf G}$ is a free $\F_p[[\mathbf G]]$-module  of rank $d(\mathbf H)$.
\end{proof}
\subsection{On amalgamated products of groups}

Let  $G$ be a group and $H_1$ and $H_2$  two subgroups that generate $G$ and have intersection $A=H_1\cap H_2$.  The following result gives an algebraic condition for $G$ to be isomorphic to the amalgamated product of $H_1$ and $H_2$ over $A$.

\begin{pro}\label{critamal} \cite{Le70}  Let $k$ be a non-trivial commutative ring. Then the canonical map $H_1*_AH_2\to G$ is an isomorphism if and only if $I_{H_1}^G\cap I_{H_2}^G=I_A^G$ in $k[G]$.

\end{pro}

\begin{proof} This is  \cite[Theorem 1]{Le70}, which is proved for $k=\Z$ but the proof works over an arbitrary nonzero commutative ring $k$.
\end{proof}

\subsection{On convergence of Sylvester  rank functions}
 Let $R$  be a ring. A {\bf Sylvester matrix rank function} $\rk$ on $R$ is a function that assigns a non-negative real number to each matrix over $R$ and satisfies the following conditions.
 \begin{enumerate}
\item [(SMat1)] $\rk(M)=0$ if $M$ is any zero matrix and $\rk(1)=1$ (where 1 denotes the identity matrix of size one);
\item [(SMat2)]  $\rk(M_1M_2) \le \min\{\rk(M_1), \rk(M_2)\}$ for any matrices $M_1$ and $M_2$ which can be multiplied;
\item[(SMat3)] $\rk\left (\begin{array}{cc} M_1 & 0\\ 0 & M_2\end{array}\right ) = \rk(M_1) + \rk(M_2)$ for any matrices $M_1$ and $M_2$;
\item[(SMat4)] $\rk \left (\begin{array}{cc} M_1 & M_3\\ 0 & M_2\end{array}\right ) \ge \rk(M_1) + \rk(M_2)$ for any matrices $M_1$, $M_2$ and $M_3$ of appropriate sizes.
\end{enumerate}
We denote by $\mathbb{P}(R)$ the set of Sylvester matrix rank functions on $R$, which is a compact convex subset of the space of functions on matrices over $R$ considered with pointwise convergence.  

Many  problems can be reinterpreted in terms of convergence in $\mathbb P(R)$. For example, if $G$ is group and $G\ge G_1\ge G_2 \ge \ldots$ is a chain of normal subgroups of $G$ of finite index with trivial intersection, then the fact of the existence of the L\"uck approximation over  $\Q$ and  $\CC$ can be viewed as    the convergence of $\rk_{G/G_i}$ to $\rk_G$ in $\mathbb P(\Q[G])$ and $\mathbb P(\CC[G])$ respectively (see \cite{Lu94, Ja19} for details).

An alternative way to introduce Sylvester rank functions is via Sylvester module rank functions.
 A {\bf Sylvester module rank function} $\dim$ on $R$ is a function that   assigns a non-negative real number to each finitely presented $R$-module   and satisfies the following conditions.
  \begin{enumerate}
\item [(SMod1)] $\dim \{0\} =0$, $\dim R =1$;
\item [(SMod2)]  $\dim(M_1\oplus M_2)=\dim M_1+\dim M_2$;
\item[(SMod3)] if $M_1\to M_2\to M_3\to 0$ is exact then
$$\dim M_1+\dim M_3\ge \dim M_2\ge \dim M_3.$$
\end{enumerate}
By \cite[Theorem 4]{Ma80} (see also \cite[Proposition 1.2.8]{LAthesis}), there exists  a natural bijection between Sylvester matrix and module rank functions over a ring. Given a Sylvester module rank function $\dim$ on $R$ and a finitely presented $R$-module $M\cong R^n/R^mA$ 
($A$ is a matrix over $R$), we define the corresponding Sylvester matrix  rank function $\rk$ by means of $\rk (A)=n-\dim M$. 

By a recent result of Li \cite{Li19}, any Sylvester module rank function $\dim$ on $R$ can be extended to a unique function (satisfying some natural conditions) on    arbitrary modules over $R$.  We will call this extension, the {\bf extended Sylvester module rank function} and denote it also by $\dim$. In this paper we will mostly use this extension for   finitely generated modules $M$. In this case $\dim M$ is defined as 
\begin{equation} \label{liext}
\dim M=\inf \{\dim \widetilde M:\ \widetilde M \textrm{\ is   finitely presented   and\ } M \textrm{\ is a quotient of\ }\widetilde M\}.\end{equation}
In the case of an arbitrary $R$-module $M$ the formula for $\dim M$ is more  complex. In this case we put
$$\dim M=\sup _{M_1} \inf _{M_2}( \dim M_2-\dim (M_2/M_1)),$$ where $M_1\le M_2$ are finitely generated $R$-submodules of $M$. Observe that we allow $+\infty$ to be a value of $\dim M$.
\begin{rem}\label{limitdim}
If $\rk$, $\rk_i\in \mathbb{P}(R)$  $(i\in \N)$ are Sylvester matrix rank functions corresponding to Sylvester module rank functions $\dim$,  $\dim_i$, respectively, then $\rk=\displaystyle \lim_{i\to \infty} \rk_i$ in the space $\mathbb{P}(R)$ if and only if for any finitely presented $R$-module $M$, $\dim M=\displaystyle  \lim_{i\to \infty} \dim_i M$. 
\end{rem}
However, the existence of the limit $\rk=\displaystyle \lim_{i\to \infty} \rk_i$ does not imply that $\dim M=\displaystyle  \lim_{i\to \infty} \dim_i M$ for any finitely generated $R$-module $M$. This phenomenon  is well-known. For example, it explains why   the L\"uck approximation of the first $L^2$-Betti numbers is valid for finitely presented groups but not always valid for finitely generated groups (see \cite{LO11}).

 If $M$ is a finitely generated $R$-module, we  only always have that  
 \begin{equation}\label{easyin}
 \dim M \ge   \limsup_{i\to \infty} \dim_i M.
 \end{equation}
 Indeed, let $\mathcal F$ be the set of all finitely presented $R$-modules $\tilde M$ such that $M$ is a quotient of $\tilde M$. Then  
\begin{multline*}
\dim M=\inf_{\tilde M\in \mathcal F}\dim \tilde M=\inf_{\tilde M\in \mathcal F}\lim_{i\to \infty} \dim_i \tilde M\ge\\ \limsup_{i\to \infty} \inf_{\tilde M\in \mathcal F}\dim_i \tilde M= \limsup_{i\to \infty} \dim_i M.\end{multline*}
In this subsection we will explain how to overcome this problem in some situations. For two Sylvester rank functions $\rk_1$ and $\rk_2\in \mathbb P(R)$ we write $\rk_1\ge \rk_2$ if $\rk_1(A)\ge \rk_2(A)$ for any matrix $A$ over $R$. If $\dim_1$ and $\dim_2$ are the Sylvester module rank functions on $R$ corresponding to $\rk_1$ and $\rk_2$, then the condition $\rk_1\ge \rk_2$ is equivalent to the condition $\dim_1\le \dim_2$, meaning that $\dim_1 M\le \dim_2 M$ for any finitely presented $R$-module $M$.

\begin{pro} \label{limitfg}Let $R$ be a ring and let $\rk$, $\rk_i\in \mathbb{P}(R)$ $(i\in \N)$ be  Sylvester matrix rank functions on $R$ corresponding to (extended) Sylvester module rank functions $\dim$,  $\dim_i$ respectively. Assume that $\rk=\displaystyle \lim_{i\to \infty} \rk_i$ and for all $i$, $\rk\ge \rk_i$. Then,  for any finitely generated $R$-module $M$, $$\dim M=\displaystyle  \lim_{i\to \infty} \dim_i M.$$
\end{pro}
\begin{proof} 
Fix $\varepsilon >0$. Let $k$ be such that
$$\liminf_{i\to \infty} \dim_i M\ge \dim_k M-\varepsilon .$$
There exists a finitely presented $R$-module  $\tilde M$ such that $M$ is a quotient of $\tilde M$ and $\dim_kM\ge \dim_k \tilde M-\varepsilon $. Since $\rk\ge \rk_k$, we have that $\dim_k \tilde M\ge \dim \tilde M$. Thus, we obtain
$$\liminf_{i\to \infty} \dim_i M\ge \dim_k M-\varepsilon \ge  \dim_k \tilde M-2\varepsilon \ge  \dim \tilde M-2\varepsilon \ge  \dim M-2\varepsilon .$$
Since $\varepsilon $ is arbitrary, we conclude that $\displaystyle \liminf_{i\to \infty} \dim_i M\ge   \dim M$. In view of (\ref{easyin}), we are done.
\end{proof} 

\subsection{Epic division $R$-rings}

Let $R$ be a ring. An {\bf epic} division $R$-ring is a $R$-ring $\phi: R\to \D$ where $\D$ is a division ring generated by $\phi(R)$.  
Moreover, we say that  $\D$ is a {\bf division $R$-ring of fractions} if $\phi$ is injective. In this case we will normally omit $\phi$ and see $R$ as a subring of $\D$.

Each epic division $R$-ring $\D$ induces a Sylvester module rank function $\dim_{\D}$ on $R$: for every a finitely  presented $R$-module $M$ we define $\dim_{\D}M$ to be equal to the dimension of $\D\otimes_R M$ as a $\D$-module. 

The extended Sylvester rank function $\dim_{\D}$ is calculated in the same way: for an $R$-module $M$, $\dim_{\D}M$ is equal to the dimension of $\D\otimes_R M$ as a $\D$-module.

 We will use $\dim_{\D}$ for the (extended) Sylvester module rank function on $R$ and for the $\D$-dimension of $\D$-spaces. This is a coherent notation because, since $\D$ is epic, $\D\otimes_R\D$ is isomorphic to $\D$ as $R$-bimodule (see \cite[Section 4]{Jasurv19}). 
 
The following result of   P. M. Cohn will be  used several times in the paper.
\begin{pro}\cite[Theorem 4.4.1]{Cohfields} \label{cohndiv} Two epic division $R$-rings $\D_1$ and $\D_2$ are isomorphic as $R$-rings if and only if $\dim_{\D_1}M=\dim_{\D_2}M$ for every finitely presented $R$-module $M$. \end{pro}

\subsection{Natural extensions of Sylvester rank functions}
Let $G$ be a group with trivial element $e$. We say that a ring $R$ is {\bf $G$-graded} if $R$ is equal to the direct sum $\oplus_{g\in G} R_g$ and $R_gR_h\subseteq R_{gh}$ for all $g$ and $h$  in $G$.  If for each $g\in G$, $R_g$ contains an invertible element $u_g$, then we say that $R$ is a {\bf crossed product} of $R_e$ and $G$ and we will write $R=S*G$ if $R_e= S$. 

Let $R=S*G$ be a crossed product. Let $\rk$ be a Sylvester matrix rank function on   $S$ and $\dim$ its associated   Sylvester module rank function. We say that $\rk$ (and $\dim$)  are {\bf $R$-compatible} if for any $g\in G$ and any matrix $A$ over $S$, $\rk(A)=\rk(u_gAu_g^{-1})$.  If $G$ is finite and $M$ is a finitely presented $R$-module, then $M$ is also  a finitely presented  $S$-module. Thus, if  $\dim$ is $R$-compatible,  we can define
\begin{equation}\label{natural}
\widetilde{\dim }\ M=\frac{ \dim M}{|G|},\end{equation}  where   $M$ is a finitely presented  $R$-module.
This defines a   Sylvester module rank function on $R$, called {\bf the natural extension} of $\dim$. This notion was  studied, for example, in  \cite{JL21}.  We notice that the same formula (\ref{natural}) holds also for extended Sylvester module rank functions (that is, when $M$ is an arbitrary $R$-module).
In this subsection we prove the following result.

\begin{pro}\label{carnatural} Let $R=S*G$ be a crossed product with $G$ finite and let $R\hookrightarrow \D$ be a division  $R$-ring of fractions. Denote by $\D_e$ the division closure of $S$ in $\D$. Then the following are equivalent.

\begin{enumerate}
\item[(a)] $\widetilde{\dim_{\D_e}}=\dim_{\D}$ as Sylvester functions on $R$.
\item [(b)] $\dim_{\D_e} \D= |G|$.
\item[(c)]$\D$ is isomorphic to a crossed product $\D_e*G$.
\item[(d)] $\D$ is isomorphic to $\D_e\otimes_ S R$ as $(\D_e , R)$-bimodule.
\item [(e)]  $\D$ is isomorphic to $R\otimes_ S \D_e$ as $(  R,\D_e)$-bimodule.
\end{enumerate}

\end{pro}
\begin{proof} For any $h,g\in G$, define $\alpha(g,h)=u_gu_h(u_{gh})^{-1}\in R$. Since the conjugation by $u_g$ fixes $S$, it also fixes $\D_e$. Therefore, we can define a ring structure on $T=\oplus_{g\in G}\D_ev_g$ defining the multiplication on homogenous elements by means of
$$ (d_1v_g)\cdot (d_2v_h)=(d_1u_gd_2(u_g)^{-1}\alpha(g,h))v_{gh}, \ d_1,d_2\in \D_e, \ g,h\in G.$$
It is clear that $T=\D_e*G$ is a crossed product, it contains  $R$ as a subring, and $T$ is isomorphic to $\D_e\otimes_ S R$ as $(\D_e , R)$-bimodule.

There exists a natural map $\gamma: T\to \D$, sending $\sum_{g\in G} d_gv_g$ ($d_g\in D_e$) to $\sum_{g\in G} d_gu_g$. Observe that $\gamma(T)$ is a domain and of finite dimension over $\D_e$. Thus, $\gamma(T)$ is a division subring of $\D$. Since $T$ contains $R$, $\D=\gamma(T)$.  This implies that (c) and (d) are equivalent. 

Since $\dim_{D_e} T=|G|$,  (b) implies that $\gamma$ is an isomorphism, and so, (b) implies (c).

Now, let us assume (d). Let $M$ be a finitely presented  $R$-module. We have the following.
\begin{multline*}
\widetilde{\dim_{\D_e}}\ M= \frac{  \dim_{\D_e}(\D_e\otimes _S M)}{|G|}=
\frac{\dim_{\D_e}(\D_e\otimes_S(R\otimes _R M))}{|G|}=\\
\frac{\dim_{\D_e}((\D_e\otimes_SR)\otimes _R M))}{|G|}\myeq{(d)}
\frac{\dim_{\D_e}(\D\otimes_R M)}{|G|}=\\ \dim_{\D} (\D\otimes_R M)=\dim_{\D} M.\end{multline*}
This proves (a).

Now, we assume that (a) holds. Since $\D_e$ es an epic $S$-ring $\D_e\otimes_S \D_e$ is isomorphic to $\D_e$ as $\D_e$-bimodule and by the same reason, $\D \otimes_R \D $ is isomorphic to $\D $ as $\D $-bimodule. Consider $M=\D$ as an $R$-module and $N=\D_e$ as a $S$-module. Then 
\begin{multline*}1=\dim_{\D}(\D\otimes_R M)=\dim_{\D} M\myeq{(a)}\widetilde{\dim_{\D_e}} M=\\
\frac{\dim_{\D_e}(\D_e\otimes_S M)}{|G|}=\frac{\dim_{\D_e}(\D_e\otimes_S(N^{\dim_{\D_e} \D}))}{|G|}=\frac{\dim_{\D_e} \D}{|G|}
\end{multline*}
This implies (b).

Let $R^{op}$ denotes the opposite ring, that is the ring with the same elements and addition operation, but with the multiplication performed in the reverse order. Then $R^{op}\cong S^{op}*G$ and the condition (c) is equivalent to
\begin{enumerate}
 \item[(c')] $\D^{op}$ is isomorphic to a crossed product $(\D_e)^{op}*G$.
 \end{enumerate}
Our previous proof gives that (c') is equivalent to  
\begin{enumerate}
 \item[(d')] $\D^{op}$ is isomorphic to $(\D_e)^{op}\otimes_ {S^{op}}R^{op}$ as $((\D_e)^{op} , R^{op})$-bimodule.
  \end{enumerate}
Now, observe  that (d') is equivalent to (e).
\end{proof}

\section{On mod-$p$ $L^2$-Betti numbers of subgroups of a free pro-$p$ groups}\label{betti}
\subsection{Universal division ring of fractions}\label{subsect:universa}

Given  two epic division $R$-rings  $\D_1$ and $\D_2$ the condition  $\dim_{\D_1}\le \dim_{\D_2}$   is equivalent to the existence of a specialization from $\D_1$ to $\D_2$ in the sense of P. Cohn (\cite[Subsection 4.1]{Cohfields}). We say that an epic division $R$-ring $\D$ is {\bf universal} if for every  division $R$-ring  $\mathcal E$, $\dim_{\D}\le \dim_{\mathcal E}$.
If a universal epic division $R$-ring exists, it is unique up to $R$-isomorphism. We will denote it by $\D_R$ and instead of $\dim_{\D_R}$ we will simply  write $\dim_R$.

We say that a ring $R$ is a {\bf semifir} if every finitely generated   left ideal of $R$ is free and free modules of distinct finite rank are non-isomorphic.  For example, if $K$ is a field, the ring $K\langle \! \langle  X\rangle \! \rangle  $ of non-commutative power series is a semifir (\cite[Theorem 2.9.4]{Cohfree}). By a theorem of  P. M. Cohn \cite{Coh74} a semifir $R$ has a universal division $R$-ring.  P. M. Cohn proved that in this case $\D_R$ can be obtained from $R$  by formally inverting all full matrices over $R$. In particular, this implies the following result.

\begin{pro} \label{internaldim}
Let $R$ be a semifir. Then
 $\dim_R$ is the smallest Sylvester module rank function among all the Sylvester module rank functions on $R$. 
 \end{pro} We will need the following result.
\begin{pro}\cite[Proposition 2.2]{Jataylor19} 
\label{vanTor} Let $R$ be a semifir. Then $$\Tor^R_1(\D_R, M)=0$$ for any $R$-submodule of a $\D_R$-module.
\end{pro}

Let $G$ be a residually torsion-free  nilpotent group (for example, $G$ is a subgroup of a free pro-$p$ group).   Let $K$ be a field. Then the universal division ring of fractions $\D_{K[G]}$ exists (see \cite{Ja20universal}). It can be constructed in the following way. Since $G$ is residually torsion free nilpotent, $G$ is bi-orderable. Fix a bi-invariant order $\preceq$ on $G$. A. Malcev \cite{Ma48} and B. Neumann \cite{Ne49} (following   an idea of H. Hahn \cite{Ha07}) showed independently that    the set $K((G,\preceq ))$ of formal power series over $G$ with coefficients in $K$ having well-ordered support has a natural structure of a ring and, moreover, it is a division ring.   $\D_{K[G]}$ can be defined as  the division closure of $K[G]$ in $K((G,\preceq ))$. The universality of this division ring is shown in \cite[Theorem 1.1]{Ja20universal}.
 
 If $A$ is a torsion-free abelian group, then $\D_{K[A]}$ coincides with the classical ring of fractions $\mathcal Q(K[A])$ of $K[A]$.

  \subsection{The   division ring $\D_{\F_p[[\mathbf F]]}$}
If  $\mathbf F$ is a  free pro-$p$ group freely generated by $f_1,\ldots, f_d$,  then, by Proposition \ref{Lazard}, the continuous homomorphism $\F_p\langle \! \langle  x_1,\ldots,x_d\rangle \! \rangle  \to \F_p[[\mathbf F]]$ that sends $x_i$ to $f_1-1$, is an isomorphism. Thus, there exists a universal division 
ring of fraction   $\D_{\F_p[[\mathbf F]]}$. Using results of  \cite{Jataylor19} we establish the following formula for $\dim_{\F_p[[\mathbf F]]}$ which is one of main ingredients of  the proof of  Theorem \ref{iteratedextension}. 
\begin{pro}\label{limitprop} 
 Let $\mathbf F=\mathbf N_1>\mathbf N_2>\mathbf N_3> \ldots$ be a chain of open normal subgroups of a finitely generated free pro-$p$ group $\mathbf F$ with trivial intersection.
 Let $M$ be a finitely generated    $\F_p[[\mathbf F]]$-module. Then
$$\dim_{\F_p[[\mathbf F]]} M= \lim_{i\to \infty}\frac{\dim_{\F_p}(\F_p[ \mathbf F/\mathbf N_i]\otimes_{\F_p[[\mathbf F]]}M )}{|\mathbf F:\mathbf N_i|}.$$
\end{pro}
 \begin{proof} 
 Let $\mathbf N$ be a normal open subgroup of $\mathbf F$ and let $\dim_{\F_p[\mathbf F/\mathbf N]}$ be  a Sylvester module rank function on $\F_p[[\mathbf F]]$ defined by 
 $$\dim_{\F_p[\mathbf F/\mathbf N]} L=\frac{\dim_{\F_p}(\F_p[ \mathbf F/\mathbf N]\otimes_{\F_p[[\mathbf F]]}L )}{|\mathbf F:\mathbf N|},$$ where
 $L$ is a finitely presented $\F_p[[\mathbf F]]$-module.
 
Let $M=\F_p[[\mathbf F]]^n/I$ be a finitely generated  $\F_p[[\mathbf F]]$-module. Since $$ {\F_p[\mathbf F/\mathbf N]} \otimes_{\F_p[[\mathbf F]]} M\cong \F[[\mathbf G]]^n/(I+(I_{\mathbf N}^{\mathbf G})^n)$$ is finite,  it is finitely presented by Lemma \ref{finitepresent}. Therefore, there exists a finitely generated $\F_p[[\mathbf G]]$-submodule $J$ of $I$ such that $I+(I_N^G)^n=J+(I_N^G)^n$.
Put $\widetilde M=\F_p[[\mathbf F]]^n/J$. Then $\widetilde M$ is a finitely presented  $\F_p[[\mathbf F]]$-module $\widetilde M$ satisfying the following conditions: 
\begin{enumerate}
\item $M$ is a quotient of $\widetilde M$ and 
 \item ${\F_p[\mathbf F/\mathbf N]} \otimes_{\F_p[[\mathbf F]]}\widetilde M\cong   {\F_p[\mathbf F/\mathbf N]} \otimes_{\F_p[[\mathbf F]]} M.$
 \end{enumerate}
 Therefore, from (\ref{liext}) we obtain that
  $$\dim_{\F_p[\mathbf F/\mathbf N]} M=\frac{\dim_{\F_p}(\F_p[ \mathbf F/\mathbf N]\otimes_{\F_p[[\mathbf F]]}\widetilde M )}{|\mathbf F:\mathbf N|}=\frac{\dim_{\F_p}(\F_p[ \mathbf F/\mathbf N]\otimes_{\F_p[[\mathbf F]]} M )}{|\mathbf F:\mathbf N|}.$$

 In the case where $M$ is finitely presented,  \cite[Theorem 1.4]{Jataylor19} and Remark \ref{limitdim} impliy that 
 \begin{equation}\label{limit}
 \dim_{\F_p[[\mathbf F]]} M= \lim_{i\to \infty} \dim_{\F_p[\mathbf F/\mathbf N_i]} M.\end{equation}
By Proposition \ref{internaldim}, we also have that $\dim_{\F_p[[\mathbf F]]} M\le \dim_{\F_p[\mathbf F/\mathbf N_i]} M$ for all $i$. Therefore, Proposition \ref{limitfg} implies that (\ref{limit}) holds also when $M$ is finitely generated. Hence,
$$ \dim_{\F_p[[\mathbf F]]} M= \lim_{i\to \infty} \dim_{\F_p[\mathbf F/\mathbf N_i]} M=\lim_{i\to \infty} \frac{\dim_{\F_p}(\F_p[ \mathbf F/\mathbf N_i]\otimes_{\F_p[[\mathbf F]]} M )}{|\mathbf F:\mathbf N_i|}.$$
\end{proof}
  In the following proposition we collect some basic properties of $\D_{\F_p[[\mathbf F]]}$.
 \begin{pro}\label{passH}
Let $\mathbf H$ be a finitely generated closed  subgroup of $\mathbf F$. The following holds.
\begin{enumerate}
 \item [(a)] Let $\D_\mathbf H$ be the division closure of $\F_p[[\mathbf H]]$ in $\D_{\F_p[[\mathbf F]]}$. Then $\D_\mathbf H$   is isomorphic to $\D_{\F_p[[\mathbf H]]}$  as an $\F_p[[\mathbf H]]$-ring. 
 \item [(b)] If  $M$ is a finitely generated $\F_p[[\mathbf H]]$-module, then
  $$\dim_{\F_p[[\mathbf H]]}(M)=\dim_{\F_p[[\mathbf F]]}(\F_p[[\mathbf F]]\otimes_{\F_p[[\mathbf H]]} M).$$
 \item[(c)]   If $\mathbf H$ is open then,  $$\D_{\F_p[[\mathbf F]]}\cong \F_p[[\mathbf F]]\otimes_{\F_p[[\mathbf H]]} \D_{\F_p[[\mathbf H]]}$$ as $(\F_p[[\mathbf F]],\F_p[[\mathbf H]])$-bimodules.
 \end{enumerate}
 \end{pro}
 \begin{proof}  (a) Fix   a normal chain $\mathbf F=\mathbf N_1>\mathbf N_2>\mathbf N_3> \ldots$  of open normal subgroups of $\mathbf F$,  and let $\mathbf H_i=\mathbf N_i\cap \mathbf H$. Let $M$ be a finitely generated $\F_p[[\mathbf H]]$-module. Observe first, that by Proposition \ref{limitprop},  
\begin{equation} \label{dimH} \dim_{\F_p[[\mathbf H]]} M= \lim_{i\to \infty}\frac{\dim_{\F_p} (\F_p[\mathbf H/\mathbf H_i]\otimes_{\F_p[[ \mathbf H]]}M) }{| \mathbf H:\mathbf H_i|}. \end{equation}

Considering  $\F_p[\mathbf F/\mathbf  N_i]$ as a right $\F_p[[\mathbf H]]$-module, we obtain that 
$$ 
\F_p[\mathbf F/\mathbf  N_i]\cong   \F_p[\mathbf H/\mathbf  H_i]^{|\mathbf F:\mathbf N_i\mathbf H|}$$ as right $\F_p[[\mathbf H]]$-modules.  Thus,
\begin{equation}\label{FtoH} \dim_{\F_p} (\F_p[\mathbf H/\mathbf H_i]\otimes_{\F_p[[ \mathbf H]]} M)=\frac{\dim_{\F_p} (\F_p[\mathbf F/\mathbf N_i]\otimes_{\F_p[[\mathbf H]]} M)}{|\mathbf  F:\mathbf N_i\mathbf H|}.\end{equation}
 Therefore, from (\ref{dimH}),  (\ref{FtoH}) and Proposition \ref{limitprop}, we conclude that
\begin{multline}\label{fphfpf}
 \dim_{\F_p[[\mathbf H]]} M=\lim_{i\to \infty}\frac{\dim_{\F_p}(\F_p[\mathbf F/\mathbf N_i]\otimes _{\F_p[[\mathbf F]]}(\F_p[[\mathbf F]] \otimes_{\F_p[[\mathbf H]]} M)}{|\mathbf F:\mathbf N_i|}=\\ 
 \dim _{\F_p[[\mathbf F]]}( \F_p[[\mathbf F]] \otimes_{\F_p[[\mathbf H]]} M).\end{multline}
 On the other hand, we have
 \begin{multline*}
\dim _{\F_p[[\mathbf F]]}( \F_p[[\mathbf F]] \otimes_{\F_p[[\mathbf H]]} M)=\\ \dim_{\D_{\F_p[[\mathbf F]]}} (\D_{\F_p[[\mathbf F]]}\otimes_{\F_p[[\mathbf F]]}( \F_p[[\mathbf F]] \otimes_{\F_p[[\mathbf H]]} M))=\\
\dim_{\D_{\F_p[[\mathbf F]]}} (\D_{\F_p[[\mathbf F]]}\otimes_{\F_p[[\mathbf H]]} M))=\dim_{\D_\mathbf H} (\D_\mathbf H \otimes_{\F_p[[\mathbf H]]} M) .\end{multline*}
Thus,  we conclude that $$\dim_{\F_p[[\mathbf H]]} M=\dim_{\D_\mathbf H} (\D_\mathbf H \otimes_{\F_p[[\mathbf H]]} M) ,$$ and so, by Proposition \ref{cohndiv}, $\D_\mathbf H$ is isomorphic to $\D_{\F_p[[\mathbf H]]}$ 
as an $\F_p[[\mathbf H]]$-ring.  

(b) This is the equality (\ref{fphfpf}).

 (c)  First assume that $\mathbf H$ is normal in $\mathbf F$. Observe that for large $i$, $\mathbf N_i\le \mathbf H$. Let $M$ be a finitely presented $\F_p[[F]]$-module. Then we obtain that
 \begin{multline*}
 \dim_{\F_p[[\mathbf F]]} M= \lim_{i\to \infty}\frac{\dim_{\F_p}(\F_p \otimes_{\F_p[[\mathbf N_i]]}M )}{|\mathbf F:\mathbf N_i|} =\\ \lim_{i\to \infty}\frac{\dim_{\F_p}(\F_p \otimes_{\F_p[[\mathbf N_i]]}M )}{|\mathbf F:\mathbf H||\mathbf H:\mathbf N_i|} =\frac{\dim_{\F_p[[\mathbf H]]} M}{|\mathbf F:\mathbf H|}.\end{multline*}
 Therefore, $ \dim_{\F_p[[\mathbf F]]} =\widetilde{ \dim_{\F_p[[\mathbf H]]} }$. Now, the result follows from Proposition \ref{carnatural}.

Now we assume that $\mathbf H$ is arbitrary,  We argue by induction on $|\mathbf F:\mathbf H|$. If $\mathbf H$ has index $p$ in $\mathbf F$, then it is normal. If  $|\mathbf F:\mathbf H|>p$, we find $\mathbf H_1$ of index $p$ in $\mathbf F$ containing $\mathbf H$. Then by induction,
\begin{multline*}\D_{\F_p[[\mathbf F]]}\cong \F_p[[\mathbf F]]\otimes_{\F_p[[\mathbf H_1]]} \D_{\F_p[[\mathbf H_1]]}\cong\\  \F_p[[\mathbf F]]\otimes_{\F_p[[\mathbf H_1]]} 
(\F_p[[\mathbf  H_1]]\otimes_{\F_p[[\mathbf H]]} \D_{\F_p[[\mathbf H]]})\cong \F_p[[\mathbf F]]\otimes_{\F_p[[\mathbf H]]} \D_{\F_p[[\mathbf H]]}.\end{multline*}
 \end{proof}
 In view of the previous proposition, we will identify  $\D_{\F_p[[\mathbf H]]}$ and  the division closure of $\F_p[[\mathbf H]]$ in $\D_{\F_p[[\mathbf F]]}$, and see $\D_{\F_p[[\mathbf H]]}$ as a subring of $\D_{\F_p[[\mathbf F]]}$.

\subsection{The division rings $\D(\F_p[G],\D_{\F_p[[\mathbf F]]})$}
Let $G$ be a subgroup of $\mathbf F$. As we have explained in Subsection \ref{subsect:universa} there exists the universal division $\F_p[G]$-ring of fractions $\D_{\F_p[G]}$. Let $\D_G=\D(\F_p[G],\D_{\F_p[[\mathbf F]]})$ be the division closure of $\F_p[G]$ in  $\D_{\F_p[[\mathbf F]]}$. 
In this subsection we will show  that $\D_{\F_p[G]}$ and $\D_G$ are isomorphic as $\F_p[G]$-rings. 
In the case $G=F$ is a finitely generated free group and $\mathbf F$ is the pro-$p$ completion of $F$, this result follows from \cite[Corollary 2.9.16]{Cohfree}.

\begin{pro} \label{divclousure}  Let $\mathbf F$ be a finitely generated free pro-$p$ group and let $G$ be a finitely generated subgroup of $\mathbf F$. 
 Let $\mathbf F=\mathbf N_1>\mathbf N_2>\mathbf N_3> \ldots$ be a chain of open normal subgroups of $\mathbf F$ with trivial intersection. We put $G_j=G\cap \mathbf N_j$. Let $\D_G=\D(\F_p[G],\D_{\F_p[[\mathbf F]]})$
  be the division closure of $\F_p[G]$ in  $\D_{\F_p[[\mathbf F]]}$. Then for every finitely generated $\F_p[G]$-module $M$,
\begin{multline*}\dim_{\D_G} ( \D_{G}\otimes_{\F_p[G]} M)= \lim_{i\to \infty}\frac{\dim_{\F_p}(\F_p \otimes_{\F_p[ G_i]}M )}{|G:G_i|}=\\
\dim_{\D_{\F_p[ G]}} ( \D_{\F_p[ G]}\otimes_{\F_p[G]} M).\end{multline*}
In particular, 
the divison rings $\D_G$ and $\D_{\F_p[G]}$ are isomorphic as  $\F_p[G]$-rings.
\end{pro}
\begin{proof} Without loss of generality we assume that $G$ is dense in $\mathbf F$.
First observe that
\begin{multline}\label{equa1}
 \dim_{\D_G} ( \D_{G}\otimes_{\F_p[G]} M)=\dim_{\D_{\F_p[[\mathbf F]]}} (\D_{\F_p[ [\mathbf F]]}\otimes_{\F_p[G]} M)=\\
\dim_{\D_{\F_p[[\mathbf F]]}} (\D_{\F_p[ [\mathbf F]]}\otimes _{\F_p[[\mathbf F]]}(\F_p[[\mathbf F]]\otimes_{\F_p[G]} M))=
\dim_{\F_p[[\mathbf F]]} (\F_p[[\mathbf F]]\otimes_{\F_p[G]} M).\end{multline}
Observe also that, since $G$ is dense in $\mathbf F$,  $|\mathbf F: \mathbf N_i|=|G:G_i|$ and
\begin{multline*}\F_p \otimes_{\F_p[ G_i]}M\cong
 \F_p[G/G_i] \otimes_{\F_p[ G]}M\cong \\
  \F_p[\mathbf F/\mathbf N_i] \otimes_{\F_p[G]}M\cong  \F_p[\mathbf F/\mathbf N_i] \otimes_{\F_p[[\mathbf F]]}(\F_p[[\mathbf F]]\otimes_{\F_p[G]} M)).\end{multline*}

Thus,  Proposition \ref{limitprop} implies that
\begin{equation}\label{equa5}  \dim_{\D_G} ( \D_{G}\otimes_{\F_p[G]} M)= \lim_{i\to \infty}\frac{\dim_{\F_p}(\F_p \otimes_{\F_p[ G_i]}M )}{|G:G_i|}.\end{equation}

Let $\mathbf F_i= \gamma_i(\mathbf F)$ and we put $H_i=G\cap \mathbf F_i$. The ring $\F_p[G/H_i]$ is a Noetherian domain and its classical field of fractions $\mathcal Q(\F_p[G/H_i])$ is universal. Moreover, by \cite[Theorem 1.2]{Ja20universal}, we have that
\begin{equation}\label{equa4} \dim_{\D_{\F_p[ G]}} ( \D_{\F_p[ G]}\otimes_{\F_p[G]} M)= \lim_{i\to \infty} \dim_{\D_{\F_p[G/H_i]}}(\D_{\F_p[G/H_i]} \otimes_{\F_p[ G]}M ).\end{equation}
 Observe that $\F_p[[\mathbf F/\mathbf F_i]]$ is also a Noetherian domain, and so the division closure of $\F_p[G/H_i]$ in   $\mathcal Q(\F_p[[\mathbf F/\mathbf F_i]])$ is isomorphic to $\D_{\F_p[G/H_i]}$ (as a $\F_p[G]$-ring). Therefore,
\begin{equation}\label{equa3} \dim_{\mathcal Q(\F_p[[\mathbf F/\mathbf F_i]])}(\mathcal Q(\F_p[[\mathbf F/\mathbf F_i]]) \otimes_{\F_p[ G]}M )=\dim_{\D_{\F_p[G/H_i]}}(\D_{\F_p[G/H_i]} \otimes_{\F_p[ G]}M ).\end{equation}
 Using Proposition \ref{limitprop}  and  arguing as in the proof of \cite[Theorem 2.3]{Ja17}, we obtain that
\begin{equation}\label{equa2} \dim_{\F_p[[\mathbf F]]} (\F_p[[\mathbf F]]\otimes_{\F_p[G]} M)= \lim_{i\to \infty} \dim_{\mathcal Q(\F_p[[\mathbf F/F_i]])}(\mathcal Q(\F_p[[\mathbf F/\mathbf F_i]]) \otimes_{\F_p[ G]}M ).\end{equation}
Therefore, putting together (\ref{equa4}), (\ref{equa3}), (\ref{equa2}), (\ref{equa1}) and (\ref{equa5}), we obtain that 
\begin{multline*}
\dim_{\D_{\F_p[ G]}} ( \D_{\F_p[ G]}\otimes_{\F_p[G]} M)=\dim_{\D_G} ( \D_{G}\otimes_{\F_p[G]} M)= \\ 
\lim_{i\to \infty}\frac{\dim_{\F_p}(\F_p \otimes_{\F_p[ G_i]}M )}{|G:G_i|}.\end{multline*}
Applying Proposition \ref{cohndiv}, we obtain that the divison rings $\D_G$ and $\D_{\F_p[G]}$ are isomorphic as  $\F_p[G]$-rings. \end{proof}

An alternative approach of proving  that $\D_G$ is isomorphic to $\D_{\F_p[G]}$ as $\F_p[G]$-ring  is taken in \cite[Lemma 7.5.5]{Mo21}, where  the result is proved by using a variation  from \cite[Theorem 6.3]{Sath09} of the uniqueness of  Hughes-free division rings \cite{Hu70} (see also \cite{DHS08}).
\subsection{Mod-$p$ $L^2$-Betti numbers}

$L^2$-Betti numbers play an important role in the solution of many problems in group theory. In the last years there was an attempt to develop a theory of mod-$p$ $L^2$-Betti numbers for different families of groups (see \cite{Jasurv19}). If $G$ is torsion-free and satisfies the Atiyah conjecture, P. Linnell \cite{Li93} showed that $L^2$-Betti numbers of $G$
 can be defined as $$b_i^{(2)}(G)= \dim_{\D(G) }H_i(G; \D(G)),$$ where $\D(G)$ is the division ring obtained as the division closure of $\Q[G]$ in the ring of affilated operators $\mathcal U(G)$. It turns out that if $G$ is residually torsion-free nilpotent, $\D(G)$ is isomorphic to the universal division ring of fractions $\D_{\Q[G]}$ (see, for example, \cite{JL19}). Therefore, we have 
$$b_i^{(2)}(G)= \dim_{\D_{\Q[G]} }H_i(G; \D_{\Q[G]}).$$
Thus, by analogy,  if  $G$ is  a residually torsion-free nilpotent group, we define the $i$th  mod-$p$ $L^2$-Betti number of $G$ as
$$\beta^{\operatorname{mod}p}_i(G)=\dim_{\D_{\F_p[G]}} H_i(G; \D_{\F_p[G]}).$$
In the case, where $G$ is a subgroup of a free pro-$p$ group, we also obtain the following formula, which can be seen as a   mod-$p$ analogue  of the L\"uck approximation theorem \cite{Lu94}.
\begin{pro} \label{beta1limit} Let $\mathbf F$ be a finitely generated free pro-$p$ group and let $G$ be a  subgroup of $\mathbf F$  of type $FP_k$ for some $k\ge 1$.  Let $\mathbf F=\mathbf N_1>\mathbf N_2>\mathbf N_3> \ldots$ be a chain of open normal subgroups of $\mathbf F$ with trivial intersection. We put $G_j=G\cap \mathbf  N_j$. Then 
$$\beta_k^{\operatorname{mod}p}(G)= \displaystyle \lim_{j\to \infty} \frac{\dim_{\F_p} H_k(G_j;\F_p)}{|G:G_j|}.$$
\end{pro}
\begin{proof}
There exists a resolution of the $\F_p[G]$-module $\F_p$ 
$$0\to R_k\to \F_p[G]^{n_{k}}\xrightarrow{\phi_k} \ldots \to \F_p[G]^{n_1}\xrightarrow{\phi_1}\F_p\xrightarrow{\phi_0} 0$$
with $R_k$ finitely generated.
The relevant  part of the sequence for calculation of $H_k(G;*)$ is the following exact sequence
$$ 0\to R_k\to \F_p[G]^{n_{k}}\to R_{k-1}\to 0,$$
($R_{k-1}=\im \phi_k=\ker \phi_{k-1}$), because for any right $\F_p[G]$-module $M$ we have
$$0\to H_k(G;M)\to M\otimes_{\F_p[G]} R_k\to  M^{n_{k}}\to M\otimes_{\F_p[G]} R_{k-1}\to 0  $$
Then we obtain that 
$$ \beta_k^{\operatorname{mod}p}(G)= \dim_{\D_{\F_p[G]}} H_k(G;\D_{\F_p[G]})=\dim_{\F_p[G]} R_k-{n_{k}}+\dim_{\F_p[G]} R_{k-1} \textrm{\ and\ }$$
$$ \dim_{\F_p} H_k(G_j;\F_p)=\dim_{\F_p}(\F_p\otimes_{\F_p[G_j]} R_k)-n_{k}|G:G_j|+ \dim_{\F_p}(\F_p\otimes_{\F_p[G_j]} R_{k-1}).$$
 
 Thus, Proposition \ref{divclousure} implies the proposition.\end{proof}
In this paper we will work only with $ \beta_1^{\operatorname{mod}p}(G)$. Observe that in this case, if $G$ is infinite,  we have   the formula
$$ \beta_1^{\operatorname{mod}p}(G)=\dim_{\F_p[G]} I_G  -1.$$
Also observe that if $A$ is a  non-trivial  torsion-free abelian group, then  since $\D_{\F_p[A]}=\mathcal Q(\F_p[A])$ is the field of fractions of $\F_p[A]$,
\begin{equation}\label{abelian} \beta_1^{\operatorname{mod}p}(A)=\dim_{\F_p[A]} I_A  -1=\dim_{\mathcal Q(\F_p[A])}(\mathcal Q(\F_p[A])\otimes_{\F_p[A]} I_A)-1=0\end{equation}
\subsection{Strong embeddings into   free  pro-$p$ groups}
Assume that a finitely generated group $G$ is a subgroup of a free pro-$p$ group $\mathbf F$. Since a closed subgroup of a free pro-$p$ group is free  (see \cite[Corollary 7.7.5]{RZ10}), changing $\mathbf F$ by the closure of $G$ in $\mathbf F$, we may assume that $G$ is dense in $\mathbf F$. 
Let $\mathbf F=\mathbf N_1>\mathbf N_2>\mathbf N_3> \ldots$ be a chain of open normal subgroups of $\mathbf F$ with trivial intersection and put $G_j=G\cap \mathbf  N_j$. 
Observe that  the closure of $G_j$ in $\mathbf F$ is equal to $\mathbf N_j$, and so
\begin{equation}\label{discpro}
|G_j:G_j^p[G_j,G_j]|\ge |\mathbf N_j:\mathbf N_j^p[\mathbf N_j,\mathbf N_j]|.
\end{equation}
Denote by $d(\mathbf F)$ the number of profinite generators of $\mathbf F$. By \cite[Theorem 7.8.1]{RZ10}, $$d(\mathbf F)= \log_p |\mathbf F:\mathbf F^p[\mathbf F,\mathbf F]|$$ and by the index Schreier formula $$d(\mathbf N_j)=(d(\mathbf F)-1)|\mathbf F:\mathbf N_j|+1.$$ Therefore, we obtain
\begin{multline*}
\dim_{\F_p} H_1(G_j;\F_p)=\log_p |G_j:G_j^p[G_j,G_j]|\ge \log_p |\mathbf N_j:\mathbf N_j^p[\mathbf N_j,\mathbf N_j]|=\\ d(\mathbf N_j)=(d(\mathbf F)-1)|\mathbf F:\mathbf N_j|+1=(d(\mathbf F)-1)|G:G_j|+1.\end{multline*}
Thus,  Proposition \ref{beta1limit} implies the following corollary.

\begin{cor} Let $G$ be a finitely generated dense subgroup of a free pro-$p$ group $\mathbf F$. Then
\begin{equation}
\label{formbeta}
\dim_{\F_pG}I_G=\beta_1^{\operatorname{mod}p}(G)+1\ge d(\mathbf F).\end{equation}
\end{cor}
This result suggests the following definition.
\begin{definition} \label{defsep}
We say that
an embedding $G\hookrightarrow \mathbf F$ of a finitely generated group $G$ into a free pro-$p$ group $\mathbf F$ is {\bf strong} if   $G$ is dense  in $\mathbf F$ and $\dim_{\F_pG}I_G=\beta_1^{\operatorname{mod}p}(G)+1 = d(\mathbf F)$.   

 A finitely generated group $G$ is called {\bf strongly embeddable in a free pro-$p$ group} (SE($p$) for simplicity) if there are  a free pro-$p$ group $\mathbf F$ and a strong embedding  $ G\hookrightarrow \mathbf F$. 
 \end{definition}
 
Let $G$ be  a parafree group. Observe that $G$ is residually-$p$ for every prime $p$. Thus, if $G$ is finitely generated, its  pro-$p$ completion  $G_{\widehat p}$ is a finitely generated free pro-$p$ group    and $G$ is a subgroup of $G_{\widehat p}$.  In this case the inequality (\ref{discpro}) is an equality, and so in the same way as we obtained the inequality (\ref{formbeta}), we obtain  that $\beta_1^{\operatorname{mod}p}(G)= d( G_{\widehat p})-1$. Therefore,   the embedding $G\hookrightarrow G_{\widehat p}$  is strong. Thus, all finitely generated parafree groups are SE($p$). On the other hand, not every subgroup of a parafree group is SE($p$). For example,   the fundamental group  of an oriented surface  of genus greater than 1   is not SE($p$).  However    the fundamental group  of an oriented surface  of genus greater than 2 can be embedded in a parafree group (see \cite[Section 4.1]{Ba69}).

By  \cite[Proposition 7.5]{BR15}, if $G$ is a finitely generated dense subgroup of a finitely generated free pro-$p$ group $\mathbf F$, then $b_1^{(2)}(G)\ge d(\mathbf F)-1$. On the other hand, by \cite[Theorem 1.6]{EL14} and Proposition \ref{beta1limit} , $b_1^{(2)}(G)\le \beta_1^{\operatorname{mod}p}(G)$. Thus, if $G\hookrightarrow \mathbf F$ is a strong embedding, we have $b_1^{(2)}(G)=\beta_1^{\operatorname{mod}p}(G)=d(\mathbf F)-1$.

\section{On $\mathcal D$-torsion-free modules.}\label{tech}

\subsection{General results}
Let $R$ be a ring and let $R\hookrightarrow \mathcal D$ be an embedding of $R$ into a division ring $\mathcal D$. Let $M$ be a left $R$-module. We say that $M$ is {\bf $\mathcal D$-torsion-free} if the canonical map $$M\to \mathcal D\otimes_R M, \ m\mapsto 1\otimes m,$$ is injective.  
The following lemma describes several   equivalent definitions of $\mathcal D$-torsion-free modules.

\begin{lem} \label{dtorsionfree}The following statements for a left $R$-module $M$ are equivalent.
\begin{enumerate}
\item[(a)] $M$ is $\mathcal D$-torsion-free.
\item[(b)] There are a $\D$-module $N$ and  an injective homomorphism $\varphi:M\to N$ of $R$-modules.
\item[(c)] For any $0\ne m\in M$, there exists a homomorphism of $R$-modules $\varphi:M\to \D$, such that $\varphi(m)\ne 0$.

\end{enumerate}

\end{lem}
\begin{proof} The proof is straightforward and we leave it to the reader.\end{proof}
Let $M$ be a left $R$-module. Recall that we use  $\dim_{\D} M $ to denote the dimension of $\D\otimes_RM$ as a $\D$-module. Observe that if  $\dim_{\D} M $ is finite, it is also equal to the dimension of $\Hom_R(M, \D)$ as a right $\mathcal D$-module.
\begin{lem} \label{dtf}
Let $0\to M_1\to M_2\to M_3\to 0$ be an exact sequence of $R$-modules. Assume that 
\begin{enumerate}
\item $M_1$ and $M_3$ are $\D$-torsion-free, 
\item $\dim_{\D} M_1$ and $\dim_{\D} M_3$ are finite and 
\item $\dim_{\D} M_2= \dim_{\D} M_1+\dim_{\D} M_3$.
\end{enumerate}
Then $M_2$ is also $\D$-torsion-free.
\end{lem}
\begin{proof} We are going to use the third characterization of $\D$-torsion-free modules from Lemma \ref{dtorsionfree}.
Consider the following exact sequence of right $\mathcal D$-modules.
$$0\to \Hom_R(M_3,\D)\to \Hom_R(M_2,\D)\to \Hom_R(M_1,\D).$$
Since $\dim_{\D} M_2= \dim_{\D} M_1+\dim_{\D} M_3$, the last map is surjective.
 
Let $m\in M_2$. If $m\not \in M_1$, then since $M_3$ is $\mathcal D$-torsion-free, there exists $\varphi \in \Hom_R(M_2/M_1, \D)$ such that $\varphi(m+M_1)\ne 0$. Hence 
there exists   $\widetilde \varphi\in  \Hom_R(M_2,\D) $, such that $\widetilde \varphi(m)\ne 0$.

If $m\in M_1$,  then since $M_1$ is $\mathcal D$-torsion-free, there exists $\varphi \in \Hom_R(M_1, \D)$ such that $\varphi(m)\ne 0$.  Using that  the restriction map 
$$ \Hom_R(M_2,\D)\to \Hom_R(M_1,\D)$$ is surjective, we obtain again that there exists $\widetilde \varphi\in  \Hom_R(M_2,\D) $, such that $\widetilde \varphi(m)\ne 0$.\end{proof}
In the calculations of $\dim_{\D}$  the following elementary lemma will be useful.

\begin{lem}\label{-1}
Let $\D$ be a division $R$-ring and $M$ be a $\D$-torsion-free $R$-module of finite $\D$-dimension.   Let $L$ be a non-trivial $R$-submodule of $M$. Then  $\dim_{\D} (M/L)<\dim_{\D} M$. Moreover, if  $\dim_{\D} L=1$, then $\dim_{\D} (M/L)=\dim_{\D} M-1$.
\end{lem}
\begin{proof}
Since $M$ is  $\D$-torsion-free, $\D\otimes_R (M/L)$ is a proper quotient of $\D\otimes_R M$. Hence $\dim_{\D} M/L<\dim_{\D} M$. 

Now assume that  $\dim_{\D} L=1$. In this case $\dim_{\D} (M/L)\ge \dim_{\D} M-1$. Therefore,  $\dim_{\D} (M/L)=\dim_{\D} M-1$.\end{proof}

\subsection{$\D_{\F_p[[\mathbf F]]}$-torsion-free modules}
Let $\mathbf F$ be a finitely generated free pro-$p$ group.
The main purpose of this subsection is to prove the following result.
\begin{pro}\label{dffree} Assume that  $1\ne z\in \mathbf F$ is not a proper $p$-power of an element of $\mathbf F$. Denote by $\mathbf Z$ the closed subgroup of $\mathbf F$ generated by $z$. Then the  module $I_{\mathbf F}/I_\mathbf Z^{\mathbf F}$ is $\D_{\F_p[[\mathbf F]]}$-torsion-free.
\end{pro}
Before proving the proposition we have to establish several preliminary results. 

\begin{lem} \label{extprop} Let $\mathbf H$ be an open   subgroup of $\mathbf F$
Let $M$ be a $\F_p[[\mathbf H]]$-module. Then $\F_p[[\mathbf F]]\otimes_{\F_p[[\mathbf H]]} M$ is $\D_{\F_p[[\mathbf F]]}$-torsion-free if and only if $M$ is $\D_{\F_p[[\mathbf H]]}$-torsion-free.

\end{lem}
\begin{proof} Assume that $M$ is $\D_{\F_p[[\mathbf H]]}$-torsion-free.
We have  that the map $ M\to \D_{\F_p[[\mathbf H]]}\otimes_{\F_p[[\mathbf H]]}M$ is injective. Then, since $\F_p[[\mathbf F]]$ is a free right $\F_p[[\mathbf H]]$-module, the map
$$\F_p[[\mathbf F]]\otimes_{\F_p[[\mathbf H]]}M\xrightarrow {\alpha}\F_p[[\mathbf F]]\otimes_{\F_p[[\mathbf H]]}( \D_{\F_p[[\mathbf H]]}\otimes_{\F_p[[\mathbf H]]} M)$$ is also injective. 

Consider  the canonical isomorphism between tensor products 
$$\F_p[[\mathbf F]]\otimes_{\F_p[[\mathbf H]]}( \D_{\F_p[[\mathbf H]]}\otimes_{\F_p[[\mathbf H]]} M)\xrightarrow {\beta} (\F_p[[\mathbf F]]\otimes_{\F_p[[\mathbf H]]} \D_{\F_p[[\mathbf H]]})\otimes_{\F_p[[\mathbf H]]} M.$$ 
By Propositopn \ref{passH}(c), 
 $$\D_{\F_p[[\mathbf F]]}\cong \F_p[[\mathbf F]]\otimes_{\F_p[[\mathbf H]]} \D_{\F_p[[\mathbf H]]}$$ 
as $(\F_p[[\mathbf F]],\F_p[[\mathbf H]])$-bimodules. Thus, there exists an isomorphism of $\F_p[[\mathbf F]]$-modules
$$ (\F_p[[\mathbf F]]\otimes_{\F_p[[\mathbf H]]} \D_{\F_p[[\mathbf H)]]})\otimes_{\F_p[[\mathbf H]]}M \xrightarrow {\gamma}  \D_{\F_p[[\mathbf F]]}\otimes_{\F_p[[\mathbf H]]} M.$$ 
We put $\varphi=\gamma\circ \beta\circ \alpha$ and apply Lemma \ref{dtorsionfree}. Since $  \D_{\F_p[[\mathbf F]]}\otimes_{\F_p[[\mathbf H]]} M$ is a $\D_{\F_p[[\mathbf F]]}$-module and $\varphi$ is an injective $\F_p[[\mathbf F]]$-homomorphism, $\F_p[[\mathbf F]]\otimes_{\F_p[[\mathbf H]]} M$ is $\D_{\F_p[[\mathbf F]]}$-torsion-free. 

Another direction of the proposition is clear because $M$ is a $\F_p[[\mathbf H]]$-submodule of $\F_p[[\mathbf F]]\otimes_{\F_p[[\mathbf H]]} M$.\end{proof}

  \begin{lem} \label{critdf} Let $\mathbf H$ be an open normal  subgroup of $\mathbf F$ and assume that $1\ne z\in \mathbf H$. Let $\mathbf Z$ be the closed subgroup of $\mathbf H$ generated by $z$. Then the following are equivalent.
  \begin{enumerate}
  \item [(a)]The  $\F_p[[\mathbf F]]$-module $ I_\mathbf H^{\mathbf F}/I^{\mathbf F}_\mathbf Z$ is not $\D_{\F_p[[\mathbf F]]}$-torsion-free.
  \item [(b)]The  $\F_p[[\mathbf H]]$-module $ I_\mathbf H/I^{\mathbf H}_\mathbf Z$ is not $\D_{\F_p[[\mathbf H]]}$-torsion-free.
    \item[(c)] There are $a\in  I_\mathbf H^{\mathbf F}$ and $b\in I_{\mathbf F}$ such that $ba=z-1$.
  \item [(d)]There are $a\in  I_\mathbf H^{\mathbf F}$ and $b\in I_{\mathbf F}$ such that $ab=z-1$.

  \end{enumerate}

\end{lem}
\begin{proof} (a)$\Longleftrightarrow$(b):  This follows from
 Lemma \ref{extprop} and  Lemma \ref{usisom}.

 (c)$\Longrightarrow$(a):  Put $
 N=\F_p[[\mathbf F]]a/I_\mathbf Z^{\mathbf F}$, Then $$N=\F_p[[\mathbf F]]a/\F_p[[\mathbf F]](z-1)=\F_p[[\mathbf F]]a/\F_p[[\mathbf F]]ba.$$  Since $b$ is not invertible in $\F_p[[\mathbf F]]$,   
 $N$ is a non-trivial submodule of $I_\mathbf H^{\mathbf F}/I_\mathbf Z^{\mathbf F}$ and since  $\F_p[[\mathbf F]]$ does not have non-trivial zero-divisors $$\F_p[[\mathbf F]]a/\F_p[[\mathbf F]]ba \cong \F_p[[\mathbf F]]/\F_p[[\mathbf F]]b.$$ Clearly
 $ \D_{\F_p[[\mathbf F]]}\otimes_{\F_p[[\mathbf F]]} N=0$, and so (a) holds.
 
 (a)$\Longrightarrow$(c):  We put $M=I_\mathbf H^{\mathbf F}/I^{\mathbf F}_\mathbf Z$ and let  $\phi: M\to \D_{\F_p[[\mathbf F]]}\otimes_{\F_p[[\mathbf F]]} M$ be the canonical map.
  Let $\overline L=L/I_\mathbf Z^{\mathbf F}$ be the kernel of $\phi$ and $\overline M$ the image of $\phi$. Hence  we have  the exact sequence
  $$0\to L\to  I_\mathbf H^{\mathbf F}\to \overline M\to 0.$$
  After applying  $\D_{\F_p[[\mathbf F]]}\otimes_{\F_p[[\mathbf F]]}$ we obtain the exact sequence 
 \begin{multline}\label{long}
 0\to   \Tor_1^{\F_p[[\mathbf F]]}(\D_{\F_p[[\mathbf F]]}, \overline M) 
\to \D_{\F_p[[\mathbf F]]} \otimes_{\F_p[[\mathbf F]]} L\to \\
    (\D_{\F_p[[\mathbf F]]})^{d(\mathbf H)}\to \D_{\F_p[[\mathbf F]]}\otimes_{\F_p[[\mathbf F]]}  \overline M\to 0.\end{multline}

  By Corrollary \ref{freemod}, the $\F_p[[\mathbf F]]$-module $I_\mathbf H^{\mathbf F}$ is free of rank $d(\mathbf H)$ and the $\F_p[[\mathbf F]]$-module $I_\mathbf Z^{\mathbf F}$ is cyclic. Thus,  by Lemma \ref{-1},
  $$ \dim_{\D_{\F_p[[\mathbf F]]}}( \D_{\F_p[[\mathbf F]]}\otimes_{\F_p[[\mathbf F]]} M)=d(\mathbf H)-1.$$
  Therefore, $$\dim_{\D_{\F_p[[\mathbf F]]}} (\D_{\F_p[[\mathbf F]]}\otimes_{\F_p[[\mathbf F]]} \overline M)=\dim_{\D_{\F_p[[\mathbf F]]}}(\D_{\F_p[[\mathbf F]]}\otimes_{\F_p[[\mathbf F]]} M)=d(\mathbf H)-1.$$
  Observe also that by Proposition \ref{vanTor}, $\Tor_1^{\F_p[[\mathbf F]]}(\D_{\F_p[[\mathbf F]]}, \overline M) =0$. Therefore,  (\ref{long}) implies that 
  $$\dim_{\D_{\F_p[[\mathbf F]]}}(\D_{\F_p[[\mathbf F]]} \otimes_{\F_p[[\mathbf F]]} L)=1.$$
By \cite[Lemma 3.1]{Ja17}  a profinite submodule of a free profinite $\F_p[[\mathbf F]]$-module is again free.   Hence $L$ is a free profinite $\F_p[[\mathbf F]]$-module, and so $L$ should be a cyclic     
  $\F_p[[\mathbf F]]$-module.  We write $L=\F_p[[\mathbf F]]a$ for some  $a\in  I_\mathbf H^{\mathbf F}$. Then there exists  $b\in \F_p[[\mathbf F]]$ such that $ba=z-1$. By our assumption $L\ne I_\mathbf Z^{\mathbf F}$. Thus, $b$ is not invertible, and so  $b\in I_{\mathbf F}$.
 
  (c)$\Longrightarrow$(d): The map $g\mapsto g^{-1}$ on $\mathbf F$ can be extended to a continuous anti-isomorphism  $\alpha: \F_p[[\mathbf F]]\to \F_p[[\mathbf F]]$. If $z-1=ba$, then $z^{-1}-1=\alpha(z-1)=\alpha(a)\alpha(b)$ and so $z-1=(-z\alpha(a))\alpha(b)$. Now note that $-z\alpha(a)\in I_\mathbf  H\F_p[[\mathbf F]]$ and $\alpha(b)\in I_{\mathbf F}$. Since $\mathbf H$ is normal in $\mathbf F$, $ I_\mathbf  H\F_p[[\mathbf F]]=I_{\mathbf H}^{\mathbf F}$, and we obtain (d).
 
(d)$\Longrightarrow$(c): It is proved in the same way as (c)$\Longrightarrow$(d).\end{proof}

Now we are ready to prove Proposition \ref{dffree}.
\begin{proof}[Proof of Proposition \ref{dffree}]  If $\mathbf F$ is cyclic, then $I_{\mathbf F}=I_{\mathbf Z}^F$ and so $I_{\mathbf F}/I_\mathbf Z^{\mathbf F}$ is $\D_{\F_p[[\mathbf F]]}$-torsion-free.

Now we assume that $\mathbf F$ is not cyclic.
There exists a normal open subgroup $\mathbf N$ of $\mathbf F$ such that $z\mathbf N$ is not a $p$-power in $\mathbf F/\mathbf N$. We will prove the proposition by induction on $|\mathbf F/\mathbf N|$.

If $\mathbf F/\mathbf N$ is cyclic, then $z\not \in \Phi(\mathbf F)$ and so $z$ is a member of a free generating system of $\mathbf F$. If $\{z,x_1,\ldots. x_k\}$ is a free generating set of $F$, then  
$$I_{\mathbf F}/I_\mathbf Z^{\mathbf F}=(\F_p[[\mathbf F]](z-1)\oplus(\oplus_{i=1}^k \F_p[[\mathbf F]](x_i-1)))/\F_p[[\mathbf F]](z-1)\cong \F_p[[\mathbf F]]^k$$ is a free $\F_p[[\mathbf F]]$-module and we are done.

Assume now that $\mathbf F/\mathbf N$ is  not cyclic. Since $\Phi(\mathbf F/\mathbf N)= \mathbf N\Phi(\mathbf F)/\mathbf N$,   the pro-$p$ group $\mathbf F/\mathbf N\Phi(\mathbf F)$ is not cyclic as well. 

Let $  \mathbf M$ be the closed  subgroup of $\mathbf F$ containing the commutator subgroup  $[\mathbf F,\mathbf F]$ and the element $z$ and such that $\mathbf M/([\mathbf F,\mathbf F]\mathbf Z)$ is the torsion part of $\mathbf F/([\mathbf F,\mathbf F]\mathbf Z)$. Since $\mathbf Z$ is cyclic and   $\mathbf F/([\mathbf F,\mathbf F])$ is torsion-free and  abelian,  $\mathbf M\Phi(\mathbf F)/\Phi(\mathbf F)$ is non-trivial cyclic  (if $z\not \in \Phi(\mathbf F)$) or trivial (if $z\in \Phi( \mathbf F)$).   Therefore, since $\mathbf F/\mathbf N\Phi(\mathbf F)$ is not cyclic, $\mathbf M\mathbf N$ 
is a proper subgroup of $\mathbf F$.

By the construction of $ \mathbf  M$, $\mathbf F/\mathbf M\cong \Z_p^k$ for some $k\ge 1$.  Since $\mathbf M\mathbf N$ 
is a proper subgroup of $\mathbf F$,
  $\mathbf M\mathbf N/\mathbf M$ is a proper subgroup of $\mathbf F/\mathbf M$. Therefore, there exists a surjective map
  $\sigma:\mathbf F\to \Z_p$   such that $\mathbf M\le \ker \sigma$ and $\mathbf N\ker \sigma\ne \mathbf F$. We put $\mathbf H=\mathbf N\ker \sigma$ and 
 extend $\sigma$  to the map $\widetilde \sigma:\F_p[[\mathbf F]]\to \F_p[[\Z_p]]$. Observe that $\ker \widetilde \sigma=I_{\ker \sigma}^\mathbf F$.
 
By way of contradiction, assume that $I_{\mathbf F}/I_\mathbf Z^{\mathbf F}$ is not $\D_{\F_p[[\mathbf F]]}$-torsion-free. Then by Lemma \ref{critdf}, there are $a,b\in I_{\mathbf F}$ such that $ab=z-1$. Applying $\widetilde \sigma$ we obtain that $\widetilde \sigma (a)\widetilde \sigma(b)=0$. Since $ \F_p[[\Z_p]]$ is a domain, either $a$ or $b$ lie in $\ker \widetilde \sigma=I_{\ker \sigma}^\mathbf F\subset I_\mathbf H^{\mathbf F}$. Applying again Lemma \ref{critdf}, we conclude that $I_\mathbf H/I_\mathbf Z^\mathbf H$ is not $\D_{\F_p[[\mathbf H]]}$-torsion-free.  

However, observe that $\mathbf N$ is also a normal subgroup of $\mathbf H$, $z\mathbf N$ is not a $p$-power in $\mathbf H/\mathbf N$ 
and $|\mathbf H/\mathbf N|< |\mathbf F/\mathbf N|$. Thus, we can apply the inductive assumption and conclude that $I_\mathbf H/I_\mathbf Z^\mathbf H$ is
  $\D_{\F_p[[\mathbf H]]}$-torsion-free.  We have arrived to a contradiction. Thus,  $I_{\mathbf F}/I_\mathbf Z^{\mathbf F}$ is  $\D_{\F_p[[\mathbf F]]}$-torsion-free. \end{proof}
\subsection{$\D_{\F_p[G]}$-torsion-free modules} 
In this subsection we assume that $\mathbf F$ is a finitely generated free pro-$p$ group and  $G$ is an (abstract) dense finitely generated  subgroup of $\mathbf F$. First we  prove the following analogue of Lemma \ref{extprop}.
\begin{lem}
 \label{extabs}  Let $H$ be a subgroup of $G$ and
let $M$ be a $\D_{\F_p[H]}$-torsion-free left $\F_p[H]$-module. Then $\F_p[G]\otimes_{\F_p[H]} M$ is $\D_{\F_p[G]}$-torsion-free.
\end{lem}
\begin{proof}   Let $\D_H$ be the division closure of $\F_p[H]$ in $\D_{\F_p[G]}$. Observe that $\D_H$ and $\D_{\F_p[H]}$ are isomorphic as  $\F_p[H]$-rings (it follows, for example, from Proposition \ref{divclousure}).

We have  that the map $ M\to \D_H\otimes_{\F_p[H]}M$ is injective. Then, since $\F_p[G]$ is a free right $\F_p[H]$-module, the map
$$\F_p[G]\otimes_{\F_p[H]}M\xrightarrow {\alpha}\F_p[G]\otimes_{\F_p[H]}( \D_H\otimes_{\F_p[H]} M)$$ is also injective. 

Consider  the canonical isomorphism between tensor products 
$$\F_p[G]\otimes_{\F_p[H]}( \D_H\otimes_{\F_p[H]} M)\xrightarrow {\beta} (\F_p[G]\otimes_{\F_p[H]} \D_H)\otimes_{\F_p[H]} M.$$ 

By \cite{Gra19}, $\D_{\F_p[G]}$ is strongly Hughes-free. This means that
 the canonical map  of $(\F_p[G],\F_p[H])$-bimodules $$\F_p[G]\otimes_{\F_p[H]} \D_H\to \D_{\F_p[G]}$$ is injective. Moreover, the image of $\F_p[G]\otimes_{\F_p[H]} \D_H$  is a direct    summand of $\D_{\F_p[G]}$ as a right $\D_H$-submodule (and so, it is also a direct  summand as a right $\F_p[H]$-submodule). Thus,  the following canonical map  of $\F_p[G]$-modules
 $$ (\F_p[G]\otimes_{\F_p[H]} \D_H)\otimes_{\F_p[H]}M \xhookrightarrow {\gamma}  \D_{\F_p[G]}\otimes_{\F_p[H]} M$$ 
 is injective. 
We put $\varphi=\gamma\circ \beta\circ \alpha$ and apply Lemma \ref{dtorsionfree}. Since $  \D_{\F_p[G]}\otimes_{\F_p[H]} M$ is a $\D_{\F_p[G]}$-module and $\varphi$ is an injective $\F_p[G]$-homomorphism, $\F_p[G]\otimes_{\F_p[H]} M$ is $\D_{\F_p[G]}$-torsion-free.

\end{proof}
   Now we can present our first example of a $\D_{\F_p[G]}$-torsion-free $\F_p[G]$-module.
\begin{pro} \label{quot} Let $H$ be a non-trivial subgroup of $G$ and $A$ a maximal abelian subgroup of $H$. Then the $\F_p[G]$-module $I^G_H/I^G_A$ is $\D_{\F_p[G]}$-torsion-free.
\end{pro}
\begin{proof}
From Lemma \ref{isomdiscr} we know that
 $$I^G_H/I^G_A\cong \F_p[G]\otimes_{\F_p[H]} (I_H/I^H_A).$$
Thus, in view of Lemma \ref{extabs}, it is enough  to show that $I_H/I^H_A$ is $\D_{\F_p[H]}$-torsion-free. 

Let $\mathbf Z=C_{\mathbf F} (A)$. Since $\mathbf F$ is a free pro-$p$ group, a centralizator of a non-trivial element is  maximal cyclic pro-$p$ subgroup. Therefore, since $A$ is abelian and non-trivial, $\mathbf Z$ is a maximal cyclic pro-$p$ subgroup of $\mathbf F$.
\begin{claim}\label{emb}
The canonical map $\F_p[H/A]\to \F_p[[\mathbf F/ \mathbf Z]]$ is injective. 
\end{claim}
\begin{proof}
 
Since $A$ is maximal abelian in $H$, we have that $A=\mathbf Z\cap H$. Hence the obvious map $\F_p[H/A]\to \F_p[\mathbf F/ \mathbf Z]$ in injective. The map  $\F_p[\mathbf F/ \mathbf Z]\to \F_p[[\mathbf F/ \mathbf Z]]$ is also injective, because $\mathbf Z$ is closed in $\mathbf F$. This finishes the proof of the claim.\end{proof}
Observe that  $\F_p[H/A]\cong \F_p[H]/I_A^H$ and $\F_p[[\mathbf F]]/I_\mathbf Z^{\mathbf F}\cong \F_p[[\mathbf F/ \mathbf Z]]$. 
Therefore,  by Claim \ref{emb}, $I_H/I^H_A$ is a $\F_p[H]$-submodule of $I_{\mathbf F}/I_\mathbf Z^{\mathbf F}$. By Proposition \ref{dffree}, we can embed $I_{\mathbf F}/I_\mathbf Z^{\mathbf F}$ in a $\D_{\F_p[[\mathbf F]]}$-module. By Proposition \ref{divclousure}, every $\D_{\F_p[[\mathbf F]]}$-module is also a $\D_{\F_p[H]}$-module. Therefore, by Lemma \ref{dtorsionfree},   $I_H/I^H_A$ is $\D_{\F_p[H]}$-torsion-free.
\end{proof}
  The following proposition shows   another  example of a $\D_{\F_p[G]}$-torsion-free $\F_p[G]$-module. This is the main result of this section.
\begin{pro}  \label{dfpfree} Let $\mathbf F$ be a finitely generated free pro-$p$ group and let $G$ be an (abstract) dense finitely generated  subgroup of $\mathbf F$.  Let $H$ be a non-trivial subgroup of $G$ and $A$ a maximal abelian subgroup of $H$. Let $B$ be an abelian subgroup of $G$ containing $A$. We put  $$J=\{(x,-x)\in I^G_H\oplus I^G_B:\ x\in I^G_A\}.$$ Then $M=(I^G_H\oplus I^G_B)/J$ is $\D_{\F_p[G]}$-torsion-free and $\dim_{\F_p[G]}M=\dim_{\F_p[G]}I^G_H$.
\end{pro}

\begin{proof} Let $L=(I^G_A\oplus I^G_B)/J\le M$. Then $L\cong I^G_B $ is a submodule of $\F_p[G]$, and so it is $\D_{\F_p[G]}$-torsion-free. The $\F_p[G]$-module $M/L$ is isomorphic to $I_H^G/I^G_A$, and so it is $\D_{\F_p[G]}$-torsion-free by Proposition \ref{quot}.

By  (\ref{abelian}), $\dim_{\F_p[A]} I_A=1$. Therefore, by Proposition \ref{passH}(b), $$\dim_{\F_p[G]}I^G_A=\dim_{\F_p[A]} I_A=1.$$ 
In the same way we obtain that $\dim_{\F_p[G]}I^G_B=1$. 

Since $\dim_{\F_p[G]}(I^G_H\oplus I^G_B)=\dim_{\F_p[G]}I ^G_H +1,$ by Lemma \ref{-1},
  \begin{multline*}
  \dim_{\F_p[G]}M=\dim_{\F_p[G]}I ^G_H +1-1=\dim_{\F_p[G]}I^G_H \textrm{\ and}\\
     \dim_{\F_p[G]} (I_H^G/I^G_A)=\dim_{\F_p[G]}I ^G_H -1= \dim_{\F_p[G]} M-1.\end{multline*}
Therefore, 
$$\dim_{\F_p[G]} (M/L)=  \dim_{\F_p[G]} (I_H^G/I^G_A) = \dim_{\F_p[G]} M-1.$$
Thus, we have obtained that $M/L$ and $L$ are $\D_{\F_p[G]}$-torsion-free and  $$\dim_{\F_p[G]}M=\dim_{\F_p[G]} (M/L)+\dim_{\F_p[G]} L.$$
Applying  Lemma \ref{dtf}, we conclude that $M$ is also $\D_{\F_p[G]}$-torsion-free.
\end{proof}

\section{Proof of main results}\label{proofs}

\subsection{The inductive step in the proof of Theorem \ref{iteratedextension}}
The following theorem is the main result of the paper.  Theorem \ref{iteratedextension} follows  from it directly.
\begin{teo} \label{main} Let $\mathbf F$  be a finitely generated free pro-$p$ group and  let $H\hookrightarrow \mathbf F$ be  a strong embedding of  a finitely generated group  $H$.  Let   $A$ be a maximal abelian subgroup of $H$  and let $B$  be  an abelian finitely generated  subgroup of $\mathbf F$ containing $A$. Put $G=\langle H, B\rangle$. Then the canonical homomorphism   $ H*_AB \to G$ is an isomorphism, and the embedding $G \hookrightarrow \mathbf F$ is strong. 
\end{teo}

\begin{proof} 
In view of Proposition \ref{critamal} we have to show that $I^G_H\cap I_B^G=I_A^G$ in $\F_p[G]$. Let 
$$J=\{(x,-x)\in I^G_H\oplus I^G_B:\ x\in I^G_A\} \textrm{\ and\ } M=(I^G_H\oplus I^G_B)/J.$$ Then  by Proposition \ref{dfpfree}, $\dim_{\F_p[G]}M=\dim_{\F_p[G]}I^G_H$. Therefore, 
 $$\dim_{\F_p[G]}M=\dim_{\F_p[G]}I^G_H=\dim_{\F_p[H]} I_H=\beta_1^{\operatorname{mod}p} (H)+1=d(\mathbf F).$$
Since $I_G=I_H^G+I_B^G$,  we have that the natural map $\alpha: M\to I_G$ is surjective. In particular $$\beta^{\operatorname{mod}p}_1(G)=\dim_{\F_p[G]}I_G-1\le\dim_{\F_p[G]} M-1= d(\mathbf F)-1.$$
Thus, using (\ref{formbeta}) we obtain that $\beta^{\operatorname{mod}p}_1(G)= d(\mathbf F)-1$ and $\dim_{\F_p[G]} I_G=\dim_{\F_p[G]}   M$. This shows that the  embedding $G \hookrightarrow \mathbf F$ is strong. 

By Proposition \ref{dfpfree}, $M$ is $\D_{\F_p[G]}$-torsion-free. Therefore, by Lemma \ref{-1}, for any proper quotient $\overline M$ of $M$, $\dim_{\F_p[G]} \overline M< \dim_{\F_p[G]}   M$. This implies that $\alpha$ is an isomorphism, and so  $I^G_H\cap I_B^G=I_A^G$. Hence, Proposition \ref{critamal} implies that  $G\cong  H*_AB$.\end{proof}

Another direct consequence  of  Theorem \ref{main} is the following corollary.
\begin{cor}\label{amalg}  Let $\mathbf F$  be a finitely generated free pro-$p$ group and  let $H\hookrightarrow \mathbf F$ be  a strong embedding of finitely generated group  $H$. 
Let $A$ be a maximal abelian subgroup of $H$. Assume that $A$ is finitely generated. Let $B$ be a finitely generated torsion-free abelian group containing $A$  and such that   $B/A$ has no $p$-torsion. Then there exists an embedding of 
$H*_AB$ into $\mathbf F$ that extends  $H\hookrightarrow \mathbf F$. In particular, $ H*_AB$ is SE($p$) (see Definition \ref{defsep}).
\end{cor}
\begin{proof} Let $H\hookrightarrow \mathbf F$ be  a strong embedding.
 Since $B/A$ has no $p$-torsion and  $B$ is torsion-free abelian, the embedding of $A$ into $\mathbf F$ can be extended to an  embedding of $B$ into $\mathbf F$. Now, we can apply Theorem \ref{main}.\end{proof}

 \subsection{$A$-groups}
 Let $A$ be a commutative ring. An {\bf $A$-group} is a group $G$ along with a map $G\times A\to G$, called ``action",  satisfying the following axioms:
 \begin{enumerate}
\item [(a)] $g^{1_A}=g$, $g^{0_A}=1$, $1^\alpha=1$;
\item [(b)] $g^{\alpha}g^{\beta}=g^{\alpha+\beta}$, $(g^\alpha)^\beta=g^{\alpha\beta}$;
\item [(c)] $(h^{-1}gh)^{\alpha}=h^{-1}g^{\alpha} h$;
\item [(d)] If $gh=hg$, $(gh)^{\alpha}=g^{\alpha}h^{\alpha}$;
  \end{enumerate}
for all $g,h\in G$ and $\alpha, \beta\in A$.  
  This definition generalizes the definition of $\Q$-group which appears in the introduction.
  
  Given a group $G$ and a commutative ring $A$, an {\bf $A$-completion} of $G$ is an $A$-group $G^A$ with a group homomorphism $\lambda: G\to G^A$ such that $G^A$ is generated by $\lambda(G)$ as an $A$-group and for any $A$-group $H$ and any homomorphism $\phi:G\to H$ there exists a unique $A$-homomorphism $\psi:G^A\to H$ such that $\phi=\psi\circ \lambda$. It is shown in  \cite[Theorems 1 and 2]{MR94} that an $A$-completion of $G$ exists and it is unique up to an $A$-isomorphism.
    
  An $A$-group $F^A(X)$ with the  set of $A$-generators $X$  is said to be a free $A$-group with base $X$, 
  if for every $A$-group $G$ an arbitrary mapping $\phi_0: X\to G$ can be extended to an $A$-homomorphism $\phi: F^A(X)\to G$.    Thus, $F^A(X)$  is the
$A$-completion of the ordinary free group $F(X)$ with free generating set $X$.  

A {\bf CSA-group} is a group in which the centralizer of every nontrivial element is abelian and malnormal. In \cite[Theorem 5]{MR96} it is shown that every torsion-free extension of centralizer of a torsion-free CSA group is again CSA. This is used  in \cite[Theorem 8]{MR96} to describe the group $F^A(X)$. In the following proposition we extract the information that we will need later.  
  \begin{pro}\label{MR}   Let $X$ be a finite set and $A$ a commutative ring  with a torsion-free additive group.  Then there are subgroups $\{W_n\}_{n\ge 0}$ of  $F^A(X)$  such that 
  \begin{enumerate}
 \item[(a)] $W_0=F(X)$;
  \item[(b)]   $W_{n+1}=\langle W_n, z_n^A\rangle$, where   $z_n\in W_n$ generates a maximal abelian subgroup in $W_n$
 and the canonical map $\displaystyle W_n*_{z_n=1_A}  A\to W_{n+1}$ is an isomorphism;
 \item[(c)]  $F(X)^A=\bigcup_{i=0}^\infty W_n$.
 \end{enumerate}
\end{pro}
 
Now we are ready to prove Corollary \ref{Zpgroup}.
  \begin{proof}[Proof of Corollary \ref{Zpgroup}]
  If $\phi$ is not injective, then there exists a finitely generated subgroup $G$ of $F^{\Z_p}(X)$ such that $G\cap \ker \phi$ is not trivial. By  Proposition \ref{MR},  there exists  a non-negative integer $k$ and a sequence $G_0\le G_1\le \ldots \le  G_k$ of subgroups of $F^A(X)$, where 
  \begin{enumerate}
  \item [(a)] $G_0=F(X_0)$ is the group generated by a finite subset $X_0$ of $X$;
  \item [(b)]   for $0\le i\le k-1$ there exists an element $z_i\in G_i$, generating a maximal abelian subgroup in $G_i$, and a  finitely generated subgroup $T_i$ of $(A,+)$ containing $1_A$, such that $z_i^{T_i}\le G_{i+1}$ and the canonical map $$G_{i}*_{z_i=1_A}T_i \to G_{i+1}$$ is an isomorphism;
  \item[(c)] $G$ is   a subgroup of $G_k$.
\end{enumerate}  Put $H_i=\phi(G_i)$ $A_i=\langle \phi(z_i)\rangle$ and $B_i=\phi(z_i^{T_i})$. 
  
  Let us show by induction on $i$ that $\phi_{|G_i}:G_i\to H_i$ is an isomorphism. It is clear for $i=0$. Assume we have proved it for $i<k$. Observe that, since $\phi(z_i)\ne 1$, $$B_i=\phi(z_i^{T_i})=\phi(z_i)^{T_i}\cong T_i \textrm{\ and \ } A_i=\phi(\langle z_i\rangle)=\phi(G_i\cap z_i^{T_i})=B_i\cap H_i.$$
  Then by Theorem \ref{iteratedextension}, $\phi_{|G_{i+1}}:G_{i+1}\to H_{i+1}$ is also an isomorphism.
  
  Thus, $\ker \phi \cap G_k=\{1\}$. This is a contradiction. 
  \end{proof}
  
We will also need the following consequence of Proposition \ref{MR}.
\begin{cor}\cite[Corollary 5]{MR96} \label{extring}
Let $B$ be  a ring with a torsion-free additive group and $A$ a subring of $B$. Then the canonical map $F^A(X)\to F^B(X)$ is injective.
\end{cor}

 \subsection{Proof of Theorem \ref{teoBau}}
  In this subsection we will prove that $F^{A}(X)$ is residually nilpotent for every commutative ring  $A$ with a torsion-free additive group.    The proof is based on the following general result communicated to us by a referee.
 \begin{pro}\label{referee}
 Let    $G$ be a group. Assume that for each $n\ge 0$ there are subgroups $G_n$  and $B_n$ of $G$ such that 
 
 \begin{enumerate}
 \item[(a)] $G_{n+1}=\langle G_n, B_n\rangle$;
 \item[(b)] $B_n$ is abelian and $G_n$ is residually torsion-free nilpotent;
 \item[(c)] $A_n=G_n\cap B_n$  and the canonical map $G_n*_{A_n} B_n\to G_{n+1}$ is an isomorphism;
 \item[(d)]  $G=\bigcup_{i=0}^\infty G_n$.
 \end{enumerate}
Then $G$ is residually torsion-free nilpotent.
 \end{pro}
 \begin{proof} Let $x\in G$. We want to show that $G$ has  a torsion-free nilpotent quotient such that the image of $x$ in this quotient is not trivial. 
 
 Let $x\in G_m$ for some $m\ge 0$. Since $G_m$  is residually torsion-free nilpotent, there exists a torsion-free nilpotent group $N$ and the map $\phi_m:G_m\to N$ such that $\phi_m(x)\ne 1$.  We can embed $N$ in its Malcev $\Q$-completion.  
 Thus, without loss of generality, we may assume that $N$ is also a $\Q$-group.
 
We will show that for each $n\ge m$ there exist $\phi_n:G_n\to N$ such that the restriction of $\phi_n$ on $G_{n-1}$ is $\phi_{n-1}$. We construct $\phi_n$ by induction. Assume that we constructed $\phi_n$ for $m\le n\le k$. Since $N$ is a $\Q$-group, we can extend $\phi_{k}$ from $A_k$ to $A_k^\Q$ and so to $B_k^\Q$ and $B_k$. Thus we have a homomorphism $\tilde \phi_k: B_k\to N$ that extends    $(\phi_k)_{|A_k}:A_k\to N$. Now, by the universal property of  the amalgamated product, there exists a homomorphism $\phi_{k+1}:G_{k+1}\to N$ whose restriction on $G_k$ is $\phi_k$ and on $B_k$ is $\tilde \phi_k$.

Let   $\phi:G\to N$ be the homomorphism satisfying  $\phi(g)=\phi_n(g)$ if $g\in G_n$. Then $\phi$ is well-defined and $\phi(x)\ne \{1\}$.
 \end{proof}
 
   \begin{cor} \label{residuallynilpotent}
  Let $X$ be a   set and $A$ a countable commutative ring  with a torsion-free additive group. Then $F^{A}(X)$ is residually torsion-free nilpotent.
   \end{cor}
  \begin{proof} First assume that $X$ is finite.
Since $A$ is countable, by Proposition \ref{MR}   there are subgroups $\{G_n\}_{n\ge 0}$  and $\{B_n\}_{n\ge 0}$ of $F^{A}(X)$ such that 
 \begin{enumerate}
 \item[(a)] $G_{n+1}=\langle G_n, B_n\rangle$;
 \item[(b)] $B_n$ is abelian and finitely generated;

 \item[(c)]  $A_n=G_n\cap B_n$  is a maximal abelian subgroup of $G_n$ and the canonical map $G_n*_{A_n} B_n\to G_{n+1}$ is an isomorphism;
 \item[(d)]  $F^A(X)=\bigcup_{i=0}^\infty G_n$.
 \end{enumerate}
By Theorem \ref{iteratedextension}, $G_n$ are residually torsion-free nilpotent.  Thus, Proposition \ref{referee} implies that $F^{A}(X)$ is residually torsion-free nilpotent.

Now assume that $X$ is an arbitrary set.  Given an element $1\ne g$ of $F^A(X)$, it belongs to $F^A(X_0)$ for some finite subset $X_0$ of $X$, and $F^A(X_0)$ is a retract of $F^A(X)$. 
As we have already proved,  there exists a homomorphism $\phi:F^A(X_0)\to N$, where $N$ is  torsion-free nilpotent such that $\phi(g)\ne 1$. This finishes the proof.
  \end{proof}
  \begin{rem} The proof of the previous corollary can also be adapted to the case where $A$ is not countable. We prove it, using a different method, in
Corollary \ref{residuallynilpotent2}.
  \end{rem}
  Let  $X=\{x_i\colon i\in I\}$ and  $Y=\{y_i\colon i\in I\}$ be two sets indexed by the elements of a set $I$. Let  
$A$ be a commutative ring. We say that $A$ is a {\bf binomial domain} if $A$ is a domain and $a \choose n$ belongs to $A$ for every $a\in A$. 
For example $\Z_p$ and $\Q$-algebras are binomial domains.

Assume that $A$ is a binomial domain. If $\Delta_A$  denotes the ideal of $A\langle \! \langle  Y\rangle \! \rangle $ generated by $Y$ then $1+\Delta_A$ is a subgroup
of the group of units of $A\langle \! \langle  Y\rangle \! \rangle $. We can define an action of $A$ on $1+\Delta_A$  in the following way:
$$(1+f)^a=1+\sum_{n=1}^\infty {a \choose n} f^n\ (a\in A, f\in \Delta_A).$$
Then, by \cite{Wa76}, $1+\Delta_A$ provided with this action is an $A$-group. Therefore, the map $x_i\mapsto 1+y_i$ can be uniquely extended to an $A$-homomorphism 
$$\phi_A: F^A(X)\to 1+\Delta_A$$ called  the {\bf Magnus representation of $F^A(X)$}. Magnus (see, for example, \cite{MKS}) proved that $\phi_\Z$ is injective.  Now we prove the  main result of this subsection.
\begin{teo}\label{MagnusA} Let $X$ be an arbitrary set and  $A$ be a binomial domain.Then 
the map $\phi_{A}: F^{A}(X)\to 1+\Delta_{A}$ is injective.
\end{teo}

\begin{proof}
Let $p$ be a prime. First, observe that by Corollary~\ref{Zpgroup} and Proposition~\ref{Lazard}, the map $\phi_{\mathbb{Z}_p}$ is injective.

Now let us consider the general case. Suppose that $\phi_A$ is not injective. Then there are a finitely generated subring $A_0 \subseteq A$ and a finite subset $X_0 \subseteq X$ such that the restriction of $\phi_A$ to $F^{A_0}(X_0)$,
\[
\phi_A|_{F^{A_0}(X_0)} \colon F^{A_0}(X_0) \to 1 + \Delta_A,
\]
is not injective.  
By \cite{Ca76}, we can embed $A_0$ into $\mathbb{Z}_p$ for some prime $p$. Let $A_1$ denote the binomial closure of $A_0$ in $\mathbb{Z}_p$, which is also isomorphic to the binomial closure of $A_0$ in $A$. Thus, $\phi_{A_1}$ is not injective.

By Proposition~\ref{extring}, we have the following embeddings:
\[
F^{A_1}(X_0) \hookrightarrow F^{\mathbb{Z}_p}(X_0).
\]
Hence, we obtain the following commutative diagram of $A_1$-groups:
\[
\begin{array}{ccc}
  F^{\mathbb{Z}_p}(X_0) & \xrightarrow{\ \phi_{\mathbb{Z}_p}\ } & 1 + \Delta_{\mathbb{Z}_p} \\
  \hookuparrow & & \hookuparrow \\
 F^{A_1}(X_0) & \xrightarrow{\ \phi_{A_1}\ } & 1 + \Delta_{A_1}
\end{array}
\]
This contradicts the assumed non-injectivity of $\phi_{A_1}$. Hence, $\phi_A$ must be injective.
\end{proof}

  The following definition has been suggested to us by a referee. We say that an $A$-group $G$  is {\bf $A$-torsion-free} if for every $1\ne g\in G$, the map $A\to G$ that sends $a$ to $g^a$ is injective.
 \begin{cor}\label{qttf} \label{residuallynilpotent2} Let $A$ be a   commutative ring  with a torsion-free additive group.
 Then the group $F^{A}(X)$ is residually-($A$-torsion-free nilpotent).
 \end{cor}
 \begin{proof}
 By Corollary \ref{extring},  we can assume that $A$ is a  $\Q$-algebra. Hence the corollary  follows from Theorem \ref{MagnusA} because the $A$-groups $1+\Delta_A/1+(\Delta_A)^n$ are $A$-torsion-free  and nilpotent.
 \end{proof}

 \subsection{Proof of Theorem \ref{Qlimit}}
The proof of Theorem \ref{Qlimit}  uses a particular case of \cite[Theorem 8]{MR96} that we describe now.
 \begin{pro}\label{QcompletionCSA}
 Let $G$ be a countable torsion-free CSA-group. Then the $\Q$-completion $G^\Q$ of $G$ is a direct union of subgroups $W_0\le W_1\le \ldots $ such that
 \begin{enumerate}
 \item 
  [(a)]$W_0=G$;
  \item [(b)]   for $i\ge 0$  there exists a  maximal abelian subgroup $A_i$ of $ W_i$   such that  $W_{i+1}$ is the image of  the canonical map $W_{i}*_{A_i}A_i^{\Q} \to G^\Q$.
 \end{enumerate}
 Moreover, if $H$ is a subgroup $G^{\Q}$ and $A$ is a maximal abelian subgroup of $H$, then  the canonical map $H*_{A}A^{\Q}\to G{^\Q}$ is injective.
 \end{pro}
 
 \begin{proof}[Proof of  Theorem \ref{Qlimit}] Our definition of a limit group from the introduction and
  Proposition \ref{MR} show that  the limit groups are exactly finitely generated subgroups of $F^{\Z[t]}(X)$. Since  $F^{\Z[t]}(X)$  is a CSA-group,
 its $\Q$-completion  can be calculated using  Proposition \ref{QcompletionCSA}. In particular,  $G^\Q$ is a subgroup of the $\Q$-completion of  $F^{\Z[t]}(X)$ and so it is a subgroup of $F^{\Q[t]}(X)$. Hence by Corollary \ref{qttf}, it is  residually torsion-free nilpotent.
\end{proof}

 \section{Linearity of free $\Q$-groups and free pro-$p$ groups}\label{linearity}
 We finish this  paper with a discussion on another two well-known problems concerning linearity of free $\Q$-groups and free pro-$p$ groups. 
 
 The problem of whether a free $\Q$-group $F^{\Q}(X)$ is linear appears in \cite[Problem F13]{BMS} and it is attributed to I. Kapovich (see also \cite[Problem 13.39(b)]{Kou}).
  The problem of whether a free pro-$p$ group $\mathbf F$ is linear is usually attributed to A. Lubotzky  (for example, we discussed this question  in Jerusalem in November,  2001). 
  
  In the context of  profinite groups, one can consider two kinds of linearity (see, for example, \cite{Ja02}). On one hand, we say that a profinite group $G$ is {\bf linear} if  it is linear as an abstract group, that is  it has a faithful representation by 
matrices of fixed degree over a field. On the other hand, the concept of {\bf $t$-linear} profinite group takes into account the topology of $G$ and means that $G$ can be faithfully  represented as a closed subgroup of the group of invertible  matrices of fixed degree over a profinite commutative ring. 

It is commonly believed that a non-abelian free pro-$p$ group is not $t$-linear (see, \cite[Conjecture 3.8]{LS94}, the discussion after \cite[Theorem 1.1]{BL99} and \cite[Section 5.3]{Sh00}). An equivalent reformulation of this statement is that a $p$-adic analytic pro-$p$ group  satisfies  a non-trvial pro-$p$ identity. A.  Zubkov \cite{Zu87}  proved  that if $p>2$, then a non-abelian free pro-$p$ group cannot be represented by 2-by-2 matrices over a profinite commutative ring.    E. Zelmanov announced that given a
fixed $n$, a non-abelian free pro-$p$ group cannot be represented by $n$-by-$n$ matrices over a profinite commutative ring   for every large enough prime
$p>>n$ (see \cite{Ze05,Ze16}). Recently, D. El-Chai Ben-Ezra, E. Zelmanov  showed that a free pro-2 group  cannot be represented by 2-by-2 matrices over a profinite commutative ring of characteristic 2 \cite{EZ19}. 

Recall that by a result of A. Malcev \cite[Theorem IV]{Ma40}, a group can be represented by matrices of degree $n$ over a field if and only if every one of its finitely generated subgroup has this property. Thus, in order to decide  whether $F^{\Q}(X)$ or  $\mathbf F$ are linear, we have to analyze the structure of their finitely generated (abstract) subgroups.  In order to apply the Malcev criterion one should find  a uniform $n$ which does not depend on a finitely generated subgroup. We may ask a weaker question of whether $F^{\Q}(X)$ and $\mathbf F$ are locally linear. Using recent advances in geometric group theory one can answer this positively in the case of $F^{\Q}(X)$.

\begin{teo} \label{linear} The groups $F^{\Q}(X)$ are locally linear over $\Z$.
\end{teo}
\begin{proof} Let $G$ be a finitely generated subgroup of $F^{\Q}(X)$. Then  by Proposition \ref{MR}, there exists $k\ge 0$ and   a sequence $G_0\le G_1\le \ldots\le  G_k$ of subgroups of $F^{\Q}(X)$,  such that $G_0$ is free, $G_{i+1}$ is obtained from $G_i$ by adjoining a root and $G\le G_k$. By \cite[Theorem 5]{MR96}, the groups $G_i$ are CSA. Thus, from a corollary on the page 100 of \cite{BF92} we obtain that $G$ is hyperbolic.   Therefore,  \cite[Corollary C]{HW15} ensures, by induction on $i$, that each $G_i$ acts properly and cocompactly on a CAT(0) cube complex. Hence, by \cite[Theorem 1.1]{Ag13},  $G_k$ has a finite index subgroup acting faithfully and specially on a CAT(0) cube complex.
 Finally from \cite[Theorem 1.1]{HW08} it follows that $G_k$, and so $G$,  are linear over $\Z$.\end{proof}
I wouldn't be surprised if the groups $F^{\Q[t]}(X)$ are also locally linear. However, to decide whether $F^{\Q}(X)$ and $F^{\Q[t]}(X)$ are linear will require a completely new approach.

\end{document}